\documentclass[a4paper,12pt]{amsart}
\usepackage{amsmath,amsthm,amssymb,amsfonts,enumerate,color,esint,graphicx, float}

\newcommand{\R}{\mathbb{R}}
\renewcommand{\L}{\mathcal{L}}

\newcommand{\Dom}{\operatorname{Dom}}

\newcommand{\dive}{\operatorname{div}}

\renewcommand{\H}{\mathcal{H}}

\newcommand{\diam}{\operatorname{diam}}
\newcommand{\tr}{\operatorname{trace}}

\renewcommand{\v}{\varphi}

\newcommand{\f}{\varphi}
\newcommand{\F}{\Phi}

\newcommand{\muf}{\mu_\f}
\newcommand{\muF}{\mu_\Phi}

\newcommand{\nf}{\nabla^\f}
\newcommand{\nF}{\nabla^\F}
\newcommand{\nFs}{\nabla^\F}
\newcommand{\dmuf}{\, d\muf}
\newcommand{\dmuF}{\,d\mu_\Phi}

\newcommand{\dx}{\, dx}
\newcommand{\dy}{\, dy}
\newcommand{\dz}{\, dz}
\newcommand{\dX}{\, dX}
\newcommand{\dY}{\, dY}
\newcommand{\na}{\mathbb{N}}
\newcommand{\naz}{\mathbb{N}_0}
\newcommand{\re}{\mathbb{R}}
\newcommand{\rn}{\mathbb{R}^n}
\newcommand{\rN}{\mathbb{R}^N}
\newcommand{\dcc}{{\rm{(DC)}}}

\newcommand{\calC}{\mathcal{C}}
\newcommand{\calS}{\mathcal{S}}

\newcommand{\eps}{\varepsilon}
\newcommand{\ep}{\epsilon}

\newtheorem{thm}{Theorem}[section]
\newtheorem{prop}[thm]{Proposition}
\newtheorem{cor}[thm]{Corollary}
\newtheorem{lem}[thm]{Lemma}

\theoremstyle{definition}
\newtheorem{defn}[thm]{Definition}
\newtheorem{rem}[thm]{Remark}

\numberwithin{equation}{section}

\allowdisplaybreaks

\author[D. Maldonado]{Diego Maldonado}
\address{Department of Mathematics\\
         Kansas State University\\
         138 Cardwell Hall, Manhattan\\
          KS 66506, USA.}
\email{dmaldona@math.ksu.edu}

\author[P. R. Stinga]{Pablo Ra\'ul Stinga}

\address{Department of Mathematics\\
	Iowa State University\\
	396 Carver Hall, Ames\\
	IA 50011, USA.}
\email{stinga@iastate.edu}

\keywords{Fractional linearized Monge--Amp\`ere equation, Harnack inequality, language of semigroups}

\subjclass[2010]{Primary: 35R09, 35R11, 35J96. Secondary: 35B65, 35J15, 47D06.}

\begin{document}

\title[Harnack for nonlocal linearized Monge--Amp\`ere]{Harnack inequality for the fractional \\
nonlocal linearized Monge--Amp\`ere equation}

\begin{abstract}
The fractional nonlocal linearized Monge--Amp\`ere equation is introduced.
A Harnack inequality for nonnegative solutions to the Poisson problem
on Monge--Amp\`ere sections is proved.
\end{abstract}

\maketitle

\section{Introduction and main results}

Throughout this paper we let $\v \in C^3(\rn)$ be a convex function with $D^2\v>0$ on $\rn$ and let $\muf$ denote its induced Monge--Amp\`ere measure 
$$
\mu_\varphi(x):=\det D^2\varphi(x).
$$
Associated to $\f$ there are three, typically degenerate/singular, elliptic operators $L_\f$, $L^\f$, and $\L_\f$ defined as
\begin{align*}
L_\f v& :=-\tr((D^2\f)^{-1}D^2v),\\
L^\f v&:= - \tr(A_\f(x) D^2v),\\
\L_\f v &: =-\dive(A_\f(x)\nabla v),
\end{align*}
where $A_\f(x)$ stands for the matrix of cofactors of $D^2\f(x)$, that is, 
\begin{equation*}\label{def:Aphi}
A_\f(x):=\mu_\f(x)(D^2\f(x))^{-1}. 
\end{equation*}
From the fact that the columns of $A_\f(x)$ are divergence-free, it follows that 
\begin{equation}\label{L=calL}
\L_\f v = L^\f v= \muf L_\f v.
\end{equation}
The elliptic equation $-L^\f u=f$ is the linearization of the Monge--Amp\`ere equation
$\det D^2 u=f$ at the function $\f$. The first identity in \eqref{L=calL} implies that
$L^\f$ admits both nondivergence (trace) and divergence (variational) forms.

In their seminal works \cite{caffaguti1, Caffarelli-Gutierrez}, L. Caffarelli and C. Guti\'errez developed a real analysis associated to $\f$ leading to their groundbreaking proof of a Harnack inequality for nonnegative solutions to $L^\f u=0$. As a crucial feature of their approach stands the description of the intrinsic geometry to study the linearized
Monge--Amp\`ere equation. This geometry is given by the Monge--Amp\`ere sections of $\f$ defined as
\begin{equation}\label{def:section}
S_\v(x_0,R):=\{x\in\R^n: \delta_\f(x_0, x) < R\},
\end{equation}
where $x_0\in\R^n$ and $R  > 0$ are called the center and the height of the section $S_\v(x_0,R)$, respectively, and
\begin{equation}\label{def:delta:f:2}
\delta_\f(x_0, x):=  \v(x) - \v(x_0) - \langle \nabla\v(x_0), x-x_0 \rangle,
\end{equation}
where $\langle \cdot, \cdot \rangle$ denotes the dot product in $\rn$. 

Notice that the case of $\f_2(x):=|x|^2/2$ accounts for the Laplacian and the Euclidean balls, since $L^{\f_2}=-\Delta$
and  $S_{\f_2}(x_0,R) = B(x_0, \sqrt{2R})$ for every $x_0 \in \rn$ and $R >0$. 

The Caffarelli--Guti\'errez regularity theory for $L^\f$ was originally motivated by its applications to fluid dynamics (see \cite[Section 1]{Caffarelli-Gutierrez} and references therein). Further applications have emerged, for instance, in relation to the affine Plateau problem in affine geometry and the prescribed affine mean curvature equation, as exposed in the work of N. Trudinger and X.-J Wang in \cite{tw02, tw05, tw00}, N. Q. Le in \cite{LeJDE16}, and references therein. 

After \cite{caffaguti1, Caffarelli-Gutierrez}, both the regularity theory
and the associated real analysis for the linearized Monge--Amp\`ere equation
have seen further progress. The Caffarelli--Guti\'errez Harnack inequality has been proved to hold under minimal geometric conditions on $\f$ in \cite{MaldonadoCVPDE14}.
It has been later extended as to allow for lower-order terms in \cite{MaldonadoJDE14} and by N. Q. Le in  \cite{LeCCM}. Interior $C^{1,\alpha}$-, $C^{2,\alpha}$- and $W^{2,p}$-estimates for solutions to $L^\f u=f$ have been established by C. Guti\'errez and T. Nguyen in  \cite{GutiNguyenAdvMath11, GutiNguyenTAMS15} and C. Guti\'errez and F. Tournier in \cite{GutiTournierTAMS06}, respectively. Global (up to the boundary) $C^{1,\alpha}$- and $W^{2,p}$-estimates have been proved by N. Q. Le and O. Savin in \cite{LeSavinARMA13, LeSavinPAMS15} and N. Q. Le and T. Nguyen \cite{LeNguyenJFA13,LeNguyenMathAnn14}, respectively. Estimates for Green's functions on Monge--Amp\`ere sections have been established in \cite{MaMRL13, MaldonadoW1p} and by N. Q. Le in  \cite{LeMaMa16}. A Liouville-type theorem for entire solutions to $L^\f u=0$ in $\re^2$ has been proved by O. Savin in \cite{Sa10}. Sobolev and Poincar\'e-type inequalities associated to $L^\f$ have been proved  in  \cite{MaldonadoCVPDE14,MaMRL13} and by G. Tian and X.-J. Wang in \cite{TianWang08}.

In this paper we develop a nonlocal version of the linearized Monge--Amp\`ere equation
and establish a Harnack inequality on Monge--Amp\`ere sections. More precisely, our purpose is to
accomplish the following goals.

\begin{enumerate}[(a)]
\item\label{goal:define:Ls:LLs} To define the fractional powers $L_\f^s$ and $\L_\f^s$
on arbitrary Monge--Amp\`ere sections (within their natural nondivergence/divergence form contexts)
and to prove existence and uniqueness of solutions to the nonlocal equations
\begin{equation}\label{eq:vfuf}
L_\f^s v=f\quad\hbox{and}\quad\L_\f^s u=F,\quad\hbox{for}~0<s<1.
\end{equation}
\item\label{goal:Ls=LLs} To show that the interplay between the (local) nondivergence and divergence
structures in \eqref{L=calL} persists on the (nonlocal) operators $L_\f^s$ and $\L_\f^s$.
\item\label{goal:H:L} To prove, under minimal geometric assumptions on $\f$, a Harnack inequality for
nonnegative solutions to \eqref{eq:vfuf} on the sections of $\f$, showing in particular that
the Monge--Amp\`ere geometry carries over to our nonlocal equations.
\end{enumerate}

The goals above could be regarded as a linear counterpart to the current efforts
to provide a correct definition of a fractional nonlocal Monge--Amp\`ere equation
by L. Caffarelli and F. Charro \cite{Caffarelli-Charro} and L. Caffarelli and L. Silvestre \cite{Caffarelli-Silvestre CAG}.

Our results will hold true for every $0 < s < 1$. Regarding \eqref{goal:define:Ls:LLs}, we should first observe that $L_\f$ is an operator in nondivergence
form. In Section \ref{Section:extension nondivergence}, we show how to define the fractional powers
$L_\f^s v$ on arbitrary Monge--Amp\`ere sections via the semigroup generated by $L_\f$.
In Section \ref{sec:example:Ls} we illustrate the definition of $L_\f^s$ by computing an explicit example of its action on the Monge--Amp\`ere quasi-distance.
Using the corresponding eigenfunctions, in Section \ref{Section:extension divergence} we define 
the fractional powers $\L_\f^s u$. Then,
in terms of existence and uniqueness of solutions to \eqref{eq:vfuf}, we have
the following result:

\begin{thm}\label{thm:existence}
Fix a section $S:=S_\f(p_0,R)$.
\begin{enumerate}[$(i)$]
\item\label{exist:L} Given any $f\in C_0(\overline{S})$ there exists a unique solution $v\in\Dom_{S}(L_\f^s)$ to
$$\begin{cases}
L_\f^sv=f,&\hbox{in}~S,\\
v=0,&\hbox{on}~\partial S.
\end{cases}$$ 
\item\label{exist:LL} Given any $F \in \Dom_{S}(\L^s_\f)'$ there exists a unique solution $u \in\Dom_{S}(\L^s_\f)$ to
$$\begin{cases}
\L_\f^su=F,&\hbox{in}~S,\\
u=0,&\hbox{on}~\partial S.
\end{cases}$$ 
\end{enumerate}
\end{thm}
Here $\Dom_{S}(L_\f^s)$ and $\Dom_{S}(\L^s_\f)$ denote the domains of $L_\f^s$ and $\L^s_\f$, with respect to the section $S$, defined in \eqref{def:DomL}--\eqref{def:DomLs} and \eqref{def:HsS}; respectively. Parts \eqref{exist:L} and \eqref{exist:LL} of Theorem \ref{thm:existence} are proved in Subsections \ref{subsec:proof:exist:L} and \ref{subsec:proof:exist:LL}, respectively. 

As far as \eqref{goal:Ls=LLs} is concerned, we show that the fractional powers $L_\f^s$ and $\L_\f^s$ do preserve the dual nondivergence/divergence nature of $L_\f$ and $\L_\f$ from \eqref{L=calL}. The
following equality is proved in Section \ref{sec:Ls=LLs}.

\begin{thm}\label{thm:L=LL}
Fix a section $S:=S_\f(p_0,R)$. Then
$$L_\f^sv = \L_\f^sv,$$
for every $v \in \Dom_S(L_\f^s)$. 
\end{thm}

Regarding \eqref{goal:H:L}, let us mention that by minimal geometric assumption on $\f$ we mean a doubling condition for $\muf$ on the sections of $\f$ known as the $\dcc_\f$-doubling condition. Namely, we write $\muf \in \dcc_\f$ if there exists a constant $C_d \geq 1$ such that
\begin{equation}\label{def:DC:intro}
\muf(S_\f(x,t)) \leq C_d \,  \muf(\tfrac{1}{2} S_\f(x,t)) \quad \forall x \in \rn, \forall t > 0,
\end{equation}
where, $\tfrac{1}{2} S_\f(x,t)$ denotes the $\tfrac{1}{2}$-contraction of $S_\f(x,t)$ with respect to its center of mass (see Section \ref{sec:background} for more about the $\dcc_\f$-doubling condition).  
Throughout the article, a \emph{geometric constant} will be a constant depending only on the $\dcc_\f$-doubling constant in \eqref{def:DC:intro}, dimension $n$, and $0<s<1$. 

\begin{thm}\label{H:nonDiver} 
Assume $\muf \in \dcc_\f$. There exist geometric constants $\kappa \in (0,1)$ and $K_9, C_H > 1 $ such that for every section $S_0:=S_\f(p_0, R_0)$, every $f\in C_0(\overline{S_0})$, every  $v\in\Dom_{S_0}(L_\f^s)$ solution to 
\begin{equation}\label{PDE:nonDiver}
\begin{cases}
L_\f^sv=f,&\hbox{in}~S_0,\\
v\geq 0,&\hbox{in}~S_0,
\end{cases}
\end{equation}
and every section $S_\f(x_0, K_9 R) \subset \subset S_0$, the following Harnack inequality holds true
\begin{equation}\label{H:v}
\sup\limits_{S_\f(x_0, \kappa R)} v \leq C_H \left( \inf\limits_{S_\f(x_0, \kappa R)} v + R^s \|f\|_{L^\infty(S_\f(x_0, K_9 R))}\right).
\end{equation}
Furthermore, there exist geometric constants $\varrho \in (0,1)$ and $K_{10} >0$ such that
\begin{equation}\label{Holder:v}
|v(x_0) - v(x)| \leq K_{10} \,  \delta_\f(x_0, x)^\varrho  \left( \sup\limits_{S_\f(x_0, K_9 R)} v + R^s \|f\|_{L^\infty(S_\f(x_0, K_9 R))}\right),
\end{equation}
for every $x \in S_\f(x_0, R)$, where $\delta_\f$ denotes the intrinsic Monge--Amp\`ere quasi-distance defined in \eqref{def:delta:f:2}. (For details on the geometric constants $\varrho$, $C_H, K_9,$ and $K_{10}$ see the proof of Theorem \ref{H:nonDiver} in Section \ref{sec:proofs:main}.)
\end{thm}

\begin{rem}
The Harnack inequality \eqref{H:v} makes a case for the central role of the Monge--Amp\`ere geometry, based on the sections $S_\f$, also in the study of the nonlocal operators $L_\f^s$. In addition, from the definition of $\Dom_{S_0}(L_\f^s)$ in \eqref{def:DomLs}, we have $v \in W^{2,n}_{\mathrm{loc}}(S_0)$ and, by the Sobolev embedding,  $v \in C^\gamma_{\mathrm{loc}}(S_0)$ for every $\gamma \in (0,1)$, where $C^\gamma_{\mathrm{loc}}(S_0)$ is the local $\gamma$-H\"older class with respect to the Euclidean distance. Now, inequality \eqref{Holder:v} says that $v \in C^\varrho_{\mathrm{loc}, \, \delta_\f}(S_0)$ with respect to the intrinsic Monge--Amp\`ere quasi-distance $\delta_\f$. 
\end{rem}

\begin{rem}
 For the particular choice $\f_2(x):=|x|^2/2$, Theorem \ref{H:nonDiver} complements,
 by also including a nonhomogeneous right hand side $f$, Theorem A from \cite{Stinga-Zhang}.
 Indeed, such a result implies a Harnack inequality for nonnegative solutions to the fractional nonlocal equation $(-\Delta_D)^s v = 0$ in a ball $B\subset\subset B_1(0)$.
 Here $-\Delta_D$ stands for the Dirichlet Laplacian in the unit ball $B_1(0)\subset\rn$.
\end{rem}
An essential tool for the proof of our main results is the extension problem
characterization of the fractional nonlocal operators $L_\f^s$ and $\L_\f^s$. The celebrated
extension problem for the fractional
Laplacian on $\rn$ was first explored from the PDE point of view
in the pioneering work of L. Caffarelli and L. Silvestre \cite{Caffarelli-Silvestre CPDE}.
This is a far reaching technique that allows to handle nonlocal problems for $(-\Delta)^s$ in a local way
through a degenerate elliptic equation in $(n+1)$-dimensions.
Later on, the semigroup language approach and the extension problem
for fractional powers of positive operators was developed in \cite{Stinga Thesis, Stinga-Torrea CPDE}.
In \cite{Stinga-Zhang} a Harnack inequality for fractional nonlocal elliptic equations admitting
variational form was proved.
The most general extension problem so far has been obtained in \cite{Gale-Miana-Stinga}. It includes
fractional powers of closed operators in Banach spaces allowing, in particular,
to deal with nonvariational equations.
Thus the results in \cite{Gale-Miana-Stinga} apply to our nondivergence form elliptic operator $L_\f$.

Theorems \ref{thm:existence} and \ref{thm:L=LL}
will be consequences of the semigroup language approach,
the localization provided by the extension problem
of \cite{Gale-Miana-Stinga, Stinga Thesis, Stinga-Torrea CPDE}
and the variational structure of $L_\v$ given by \eqref{L=calL}.

The main steps in the proof of Theorem \ref{H:nonDiver} are as follows.
First, given $f\in C_0(\overline{S_0})$ and
the nonnegative solution $v\in\Dom_{S_0}(L_\f)$, we will establish the
equivalence between the fractional nonlocal equation \eqref{PDE:nonDiver}
and the local degenerate/singular extension problem in one more variable
\begin{equation}\label{eq:extension nondivergence form:intro}
\begin{cases}
-L_\f V+z^{2-1/s}V_{zz}=0,&\hbox{for}~x\in S_0,~z>0,\\
V(x,z)=0,&\hbox{for}~x\in\partial S_0,~z\geq0, \\
\displaystyle\lim_{z\to0^+}V(x,z)=v(x),&\hbox{uniformly in}~S_0,\\
\displaystyle-\lim_{z\to0^+}V_z(x,z)=d_sf(x),&\hbox{uniformly in}~S_0,
\end{cases}
\end{equation}
as shown in \cite{Gale-Miana-Stinga}. Here $d_s > 0$ is an explicit constant defined in \eqref{def:ds}.
There is a unique nonnegative solution $V$, vanishing at infinity in the sense of \eqref{eq:uniformly to zero}, such that
$$
V\in C^\infty((0,\infty);\Dom_{S_0}(L_\f))\cap \,C^1([0,\infty);C_0(\overline{S_0})).
$$
Second, by setting $\widetilde{V}(x,z)  := V(x,|z|)$ for every $(x,z)$ in the cylinder $S_0 \times \re$  we have that $\widetilde{V} \in C^2(\re \setminus\{0\};\Dom_{S_0}(L_\f))\cap \mathrm{Lip}(\re;C_0(\overline{S_0}))$ solves 
\begin{equation}\label{PDE:tildeV}  
-L_\f \widetilde{V} + |z|^{2-1/s} \widetilde{V}_{zz}  =0  \quad \text{pointwise in } S_0 \times (\re \setminus \{0\}).
\end{equation}
Equation \eqref{PDE:tildeV} can be recast as the linearized Monge--Amp\`ere equation 
\begin{equation}\label{LMA:Vtilde}
L_\F(\widetilde{V}) :=  -\tr{( (D^2 \F)^{-1} D^2 \widetilde{V}})  = 0, \quad \text{pointwise in } S_0 \times (\re \setminus \{0\}),
\end{equation}
where $\F \in C^1(\re^{n+1})$ is the strictly convex function (recall that $1/s>1$) defined as
\begin{equation*}
\F(x,z):= \f(x) + \tfrac{s^2 }{(1-s)} |z|^{1/s}, \quad \hbox{for all}~(x,z) \in \rn \times \re.
\end{equation*}
In Section \ref{sec:S(Q)} we set $Q:=S_0 \times \re$ and define a functional class $\calS(Q)$ that contains $\widetilde{V}$. Then, Theorem \ref{H:nonDiver} will result from \eqref{eq:extension nondivergence form:intro} and the following Harnack inequality (see Section \ref{sec:S(Q)} for notation).

\begin{thm}\label{thm:H:Vtilde} 
Assume $\muf \in \dcc_\f$. Then, there exist geometric constants $\kappa \in (0,1)$ and $K_7, \widetilde{C}_H > 1$ such that for every nonnegative $W \in \calS(Q)$ solution to $L_\F(W)=0$ pointwise in $Q^+$ and every section $S_\F(X_0, R)$ with 
$$
S_\F(X_0, K_7 R) \subset \subset Q
$$
the following Harnack inequality holds true
\begin{equation}\label{H:W}
\sup\limits_{S_\F(X_0, \kappa R)} W \leq \widetilde{C}_H \left( \inf\limits_{S_\F(X_0, \kappa R)} W + R^s W_{z,0^+}(S_\F(X_0,K_7 R))\right).
\end{equation}
Here $W_{z,0^+}(S_\F(X_0,K_7 R))$ stands for the $L^\infty$-norm of the normal derivative of $W$ on the intersection $S_\F(X_0,K_7 R) \cap  \{(x, z) \in \re^{n+1}: z =0\}$, as defined  in \eqref{def:F0+}. 

Consequently, there exist geometric constants $\varrho \in (0,1)$ and $K_{11} >0$ such that
\begin{equation}\label{Holder:W}
|W(X_0) - W(X)| \leq K_{11}  \delta_\F(X_0, X)^\varrho \left(\sup\limits_{S_\F(X_0,R)} W + R^sW_{z,0+}(S_\F(X_0,K_4R))\right),
\end{equation}
for every $X \in S_\F(X_0, R)$. 
\end{thm}
 
The major obstacles in the proof of Theorem \ref{thm:H:Vtilde}  are the following:

\begin{enumerate}[(i)]
\item  On the hyperplane $\{(x, z) \in \re^{n+1}: z =0\}$, the matrix 
\begin{equation*}
D^2 \Phi(x,z)^{-1}=
\begin{pmatrix}
D^2\v(x)^{-1} & 0 \\
0 & |z|^{2-1/s}
\end{pmatrix}
\in \R^{n+1}\times\R^{n+1}
\end{equation*}
becomes degenerate (if $1/2<s<1$) or singular (if $0<s<1/2$). This degeneracy/singularity prevents the direct application (even under stronger assumptions than $\muf \in \dcc_\f$) of known Harnack inequalities for the linearized Monge--Amp\`ere equation including  \cite[Theorem 5]{Caffarelli-Gutierrez}, \cite[Theorem 1]{LeCCM},  \cite[Theorem 1.4]{MaldonadoCVPDE14}, and \cite[Theorem 1]{MaldonadoJDE14}, which require continuous second-order derivatives as well as positive-definite Hessians for the underlying convex function. 
\item Furthermore, the fact that $\det D^2 \F$ may vanish implies that the Caffarelli--Guti\'errez proof of the ``passage to the double section'' (that is, \cite[Theorem 2]{Caffarelli-Gutierrez}, which is essential for the weak-Harnack inequality in \cite[Theorem 4]{Caffarelli-Gutierrez}) cannot be applied to \eqref{LMA:Vtilde}.
\item Also, since $(D^2 \F)^{-1}$ becomes degenerate or singular, the solvability of the Dirichlet problem for $L_\F$
in \eqref{PDE:tildeV} cannot be taken for granted and then the argument from \cite{MaldonadoJDE14} used to prove the critical-density estimate (see  \cite[Theorem 2]{MaldonadoJDE14}) cannot be carried out for  \eqref{LMA:Vtilde}.
\item The function $\widetilde{V}$ does not belong apriori to $W^{2,n+1}_{\mathrm{loc}}(S_0 \times \re)$, which prevents its direct use in the Aleksandrov--Bakelman--Pucci maximum principle. 
\end{enumerate}

Notice that the hyperplane $\{(x, z) \in \re^{n+1}: z =0\}$ is central to our approach because it is precisely there
 where the connection between $\widetilde{V}$ and $v$ takes place, see \eqref{eq:extension nondivergence form:intro}. 

In addition to the obstacles above, we pursue (and prove) Theorem \ref{thm:H:Vtilde}  under minimal geometric conditions, that is, under $\muf \in \dcc_\f$ only. This is in contrast with \cite{Caffarelli-Gutierrez}.

In order to overcome those obstacles, in Section \ref{sec:CD} we show that, despite the degeneracy or singularity of $(D^2 \F)^{-1}$ and under the hypothesis $\muf \in \dcc_\f$ only, a critical-density estimate for $\widetilde{V}$ (Theorem \ref{thm:CD} below) can be established, for \emph{some} $\eps_0 \in (0,1)$, in terms of the $L^\infty(S_0)$-norm of the normal derivative $\displaystyle -\lim_{z\to0^+}\widetilde{V}_z(x,z)=d_sf(x)$. That is, a suitable control on the normal derivative will allow for a critical-density estimate in the absence of the hypothesis of continuous second-order derivates and positive-definite Hessian. This represents a novelty in the study of the linearized Monge--Amp\`ere equation.

With a critical-density estimate at hand, mean-value inequalities (see Theorems \ref{thm:LB:q=1} and \ref{thm:LB}) can be obtained by fairly standard arguments, as described in Section \ref{sec:LB}.  

The next step is to prove a weak-Harnack inequality for $\widetilde{V}$. As mentioned, the methods from \cite[Sections~2~and~3]{Caffarelli-Gutierrez} are not applicable. Instead, we rely on the variational side of the linearized Monge--Amp\`ere equation. 

The idea is the following. By \eqref{L=calL}, we have $\muf L_\f = \L_\f$, and we know that $V=V(x,z)$ comes from the extension problem \eqref{eq:extension nondivergence form:intro} associated to $L_\f$. Now, by developing the extension problem associated to the divergence-form operator $\L_\f$ from \cite{Caffarelli-Stinga, Stinga Thesis, Stinga-Torrea CPDE}
(namely, \eqref{eq:extension divergence form} below) we find a solution $U=U(x,y), x \in S_0, y \in \re$.  Certainly, after contrasting the extension problems associated to $L_\f$ and $\L_\f$, one cannot expect that $U= V$. However, in Section \ref{sec:Ls=LLs} we prove that, modulo a change of variables (more precisely, a change in the extension variables $y$ and $z$), the equality $U=V$ does hold true. A key consequence of this equality is the energy estimate for $V$ given by \eqref{E:E:V} in Remark \ref{rmk:finite:E}. Although $V$ solves the PDE in \eqref{eq:extension nondivergence form:intro} which can also be written in divergence form, we were able to arrive at the energy estimate \eqref{E:E:V} only by going through $U$. 

With this insight on the variational side of $V$, in Section \ref{sec:PI:F} we develop a weak Poincar\'e inequality associated to $\F$. One advantage of working in the variational context (which is based on energy estimates) is that the pointwise singularity of $D^2 \Phi^{-1}$ will become inconsequential because $|z|^{1/s-2} \in L^1_{\mathrm{loc}}(\re)$ for every $0<s<1$. In turn, in Section \ref{sec:every:CD} we use the Poincar\'e inequality to prove a critical-density estimate for \emph{every} $\eps \in (0,1)$ (see Theorem \ref{thm:every:CD}). In Section \ref{sec:weak:H} we use the fact that the critical density can be taken small enough and combine it with a covering lemma to prove a weak-Harnack inequality for $\widetilde{V}$ (see Theorem \ref{thm:weak:H}). Finally, in Section \ref{sec:proofs:main} we bring all the previous results together to prove Theorems \ref{H:nonDiver} and \ref{thm:H:Vtilde}. 

We close this introduction by mentioning that no use of the normalization technique from \cite[Section 1]{Caffarelli-Gutierrez} or of the local Monge--Amp\`ere-$BMO$ space (which dominated the variational approach in \cite{MaldonadoJDE14, MaldonadoCVPDE14}) is made in this article. 

\section{nondivergence form:  fractional powers $L_\f^s$ and extension problem}\label{Section:extension nondivergence}

Fix any section $S:=S_\f(x_0, R)$. On $S$ we consider the linearized Monge--Amp\`ere equation with
homogeneous Dirichlet boundary condition:
\begin{equation}\label{eq:problem1}
\begin{cases}
L_\f v\equiv-\tr((D^2\f)^{-1}D^2v)=f,&\hbox{in}~S,\\
\, v=0,&\hbox{on}~\partial S.
\end{cases}
\end{equation}

\subsection{The semigroup generated by $L_\f$}

Our first goal is to introduce the semigroup generated by $L_\f$.
The matrix of coefficients $(D^2\f)^{-1}$ is symmetric and positive definite, with entries in $C^1(\overline{S})$.
Since $\overline{S}$ is a compact set, such a matrix is uniformly elliptic on $S$.
Notice that we use this fact without ever resorting to estimates depending
on the size of the eigenvalues of $D^2\v(x)^{-1}$ for $x \in \overline{S}$.
Let the domain of $L_\f$ be the space
\begin{equation}\label{def:DomL}
\Dom_S(L_\f):=\big\{v\in C_0(\overline{S})\cap W^{2,n}_{\mathrm{loc}}(S):L_\f v\in C(\overline{S})\big\},
\end{equation}
where $C_0(\overline{S})$ is the Banach space
$C_0(\overline{S}):= \big\{ v \in C(\overline{S}) : v =0 \text{ on  } \partial S\big\}$
endowed with the $L^\infty(S)$ norm. Here as usual $C(\overline{S})$ denotes the space of continuous
functions on $\overline{S}$ under the $L^\infty(S)$ norm.
It is clear that $\Dom_S(L_\f)$ depends on the section $S$. Also, let us define $\Dom_S(L_\f^s)$ as
\begin{equation}\label{def:DomLs}
\Dom_S(L_\f^s):=\Dom_S(L_\f).
\end{equation}
It follows from \cite[Theorem~4.1]{Arendt-Schatzle} that $(L_\f,\Dom_S(L_\f))$ generates a
bounded holomorphic semigroup $\{e^{-tL_\f}\}_{t\geq0}$ on $C(\overline{S})$.
For convenience we recall the relevant definitions, see for example \cite{Arendt, Pazy, Yosida}.
The family $\{e^{-tL_\f}\}_{t\geq0}$ is a \emph{semigroup} on $C(\overline{S})$ (see \cite[Section~1.1]{Pazy})
if the following conditions hold:
\begin{enumerate}[(i)]
\item for each $t\geq0$, $e^{-tL_\f}$ is a bounded linear operator from $C(\overline{S})$
into itself;
\item the semigroup property holds: for every $t_1,t_2\geq0$ and for any $v\in C(\overline{S})$,
$$e^{-t_1L_\f}(e^{-t_2L_\f}v)=e^{-(t_1+t_2)L_\f}v;$$
\item for every $v\in C(\overline{S})$ we have $e^{-0L_\f}v=v$.
\end{enumerate}
The semigroup $e^{-tL_\v}$ is \textit{bounded holomorphic} if the operator valued
function $t\to e^{-tL_\f}$ from $[0,\infty)$ into the
algebra of bounded linear operators on $C(\overline{S})$ has a holomorphic extension
to an open sector of the complex plane contained in $\operatorname{Re} z>0$,
which is bounded on proper subsectors, see \cite[p.~150]{Arendt}.
It is shown in \cite[Proposition~4.4]{Arendt-Schatzle} that $(L_\f,\Dom_S(L_\f))$ is
dissipative, see also \cite[Lemma~3.4.2]{Arendt}.
Then, as a consequence of \cite[Proposition~3.7.16]{Arendt}, we obtain that
the semigroup is a \textit{contraction}. In other words, for every $v\in C(\overline{S})$,
\begin{equation}\label{eq:contraction}
\|e^{-tL_\f}v\|_{L^\infty(S)}\leq\|v\|_{L^\infty(S)},\quad\hbox{for all}~t\geq0.
\end{equation}
The fact that $(L_\f,\Dom_S(L_\f))$ is the \emph{generator} of the semigroup $e^{-tL_\f}$ means that
for every $v\in\Dom_S(L_\f)$ the function $w(t,x):=e^{-tL_\f}v(x)$ is the unique solution to the parabolic equation
$$\begin{cases}
\displaystyle\partial_tw=-L_\f w,&\hbox{for}~t>0,x\in S,\\
\displaystyle w(t,x)=0,&\hbox{for}~t\geq0,x\in\partial S,\\
\displaystyle\lim_{t\to0^+}w(t,x)=v(x),&\hbox{uniformly on}~S,
\end{cases}$$
see \cite[Sections~1.1~and~1.2]{Pazy}.
Observe that the semigroup $e^{-tL_\f}$, though a contraction, is not a $C_0$-semigroup on $C(\overline{S})$
because $\Dom_S(L_\f)$ is not dense in $C(\overline{S})$,
see \cite[Hille--Yosida~Teorem~in~Section~1.3]{Pazy}. In particular,
in order for $w(t,x)$ above to converge to the initial data $v(x)$
uniformly in $S$ we need to take $v\in\Dom_S(L_\f)$. As $e^{-tL_\f}$
is a bounded holomorphic semigroup, from \cite[Theorem~3.7.19]{Arendt} we have that
$$v\in C(\overline{S})\quad\hbox{implies}\quad e^{-tL_\f}v\in\Dom_S(L_\f),~\hbox{with}~
\sup_{t>0}\|tL_\f e^{-tL_\f}v\|_{L^\infty(S)}<\infty.$$
Finally, the following \textit{decay estimate} holds: there are constants $M,\gamma>0$ such that,
for every $v\in\Dom_S(L_\f)$,
\begin{equation}\label{eq:exponential decay semigroup}
\|e^{-tL_\v}v\|_{L^\infty(S)}\leq Me^{-\gamma t}\|v\|_{L^\infty(S)},\quad\hbox{for all}~t\geq0,
\end{equation}
see \cite[Theorem~1]{Oddson}, also \cite[Theorem~4.1,~Proposition~4.7]{Arendt-Schatzle}.

\begin{rem}[Positivity]\label{rem:positive}
It is important to notice that the semigroup $e^{-tL_\f}$ is \textit{positive} on $\Dom_S(L_\f)$.
Namely, if $v\in\Dom(L_\f)$ and $v\geq0$ in $S$ then $e^{-tL_\f}v(x)\geq0$,
for every $x\in S$ and $t\geq0$. Indeed, this follows from the well known weak minimum principle
for parabolic equations in nondivergence form.
\end{rem}

\subsection{The fractional nonlocal operator $L^s_\f$}

The semigroup generated by $L_\f$ allows us to define the fractional powers $L_\f^sv(x)$
as in \cite[p.~260,~(5)]{Yosida} and \cite[Theorem~4.1]{Gale-Miana-Stinga}.

\begin{defn}\label{definition:fractional nondivergence}
Let $0<s<1$. The fractional operator $L_\f^sv(x)$ is defined for any $v\in\Dom_S(L_\f)$
and every $x\in S$ as
\begin{equation}\label{eq:power}
L_\f^sv(x)=\frac{1}{\Gamma(-s)}\int_0^\infty\big(e^{-tL_\f}v(x)-v(x)\big)\,\frac{dt}{t^{1+s}}.
\end{equation}
\end{defn}

The integral in \eqref{eq:power} is absolutely convergent in the sense of Bochner. Indeed,
the integrand $(e^{-tL_\f}v-v)t^{-(1+s)}$ is a function of $t\in(0,\infty)$
with values in $C_0(\overline{S})$ and, since $v\in\Dom_S(L_\f)$ and $e^{-rL_\f}$ is a contraction,
we have (see \cite[p.5,~(2.5)--(2.6)]{Pazy}),
$$\|e^{-tL_\f}v-v\|_{L^\infty(S)}\leq \int_0^t\|\partial_re^{-rL_\f}v\|_{L^\infty(S)}\,dr=
\int_0^t\|e^{-rL_\f}L_\f v\|_{L^\infty(S)}\,dr\leq\|L_\f v\|_{L^\infty(S)}t.$$
On the other hand, by contractivity,
$\|e^{-tL_\f}v-v\|_{L^\infty(S)}\leq 2\|v\|_{L^\infty(S)}$, for every $t\geq0$. Therefore,
for any $A>0$,
$$\int_0^\infty\|e^{-tL_\f}v-v\|_{L^\infty(S)}\,\frac{dt}{t^{1+s}}\leq \|L_\f v\|_{L^\infty(S)}\int_0^A\,\frac{dt}{t^s}+
2\|v\|_{L^\infty(S)}\int_A^\infty\,\frac{dt}{t^{1+s}}<\infty.$$
In particular, the following fractional Sobolev-type interpolation inequality holds
$$\|L^s_\f v\|_{L^\infty(S)}\leq\frac{sA^{1-s}}{\Gamma(2-s)}\|L_\f v\|_{L^\infty(S)}+
\frac{2}{A^s\Gamma(1-s)}\|v\|_{L^\infty(S)} .$$
for any $v\in\Dom_S(L_\f)$, $A>0$ and $0<s<1$, and, as a consequence,
$$\|L^s_\f v\|_{L^\infty(S)}\leq\frac{2^{1-s}}{\Gamma(2-s)}\|L_\f v\|^s_{L^\infty(S)}\|v\|^{1-s}_{L^\infty(S)}.$$

We finally notice that, unlike the local differential operator $L_\f$ that in general has values in $C(\overline{S})$,
see \eqref{def:DomL}, the fractional nonlocal operator $L^s_\f$ has range in $C_0(\overline{S})$, namely,
$$v\in\Dom_S(L_\f^s)\quad\hbox{implies}\quad L_\f^sv\in C_0(\overline{S}).$$

\begin{rem}[Maximum principle]
The semigroup expression for $L^s_\f v(x)$ in \eqref{eq:power} yields the
following maximum principle. Let $v\geq0$ and suppose that $v(x_0)=0$ for some $x_0\in S$.
Then $L^s_\f v(x_0)\leq 0$ in $S$. Moreover, $L^s_\f v(x_0)=0$ if and only if $v\equiv0$ in $S$.
Indeed, since the semigroup is positive
(see Remark \ref{rem:positive}) we have $e^{-tL_\f}v(x_0)\geq0$ and
the first conclusion follows from \eqref{eq:power} by noticing that $\Gamma(-s)<0$.
Furthermore, $L_\v^s v(x_0)=0$ if and only if $e^{-tL_\v}v(x_0)=0$ for all $t>0$,
which in this case is equivalent to $v\equiv0$ in $S$
by the usual weak maximum principle for parabolic equations.
\end{rem}

\subsection{Proof of Theorem \ref{thm:existence}(i)}\label{subsec:proof:exist:L}

Given $f\in C_0(\overline{S})$, it follows from general theory (see \cite[Section~3.8]{Arendt} and
\cite[Chapter~IX,~Section~11]{Yosida}) that
$$
v(x):=L_\v^{-s}f(x)=\frac{1}{\Gamma(s)}\int_0^\infty e^{-tL_\f}f(x)\,\frac{dt}{t^{1-s}},\qquad \forall x\in S,
$$
is a solution to $L_\f^sv=f$ in $S$ with $v=0$ on $\partial S$. To prove uniqueness,
assume that $v\in\Dom_S(L^s_\f)$ is a non trivial solution to $L^s_\f v=0$ in $S$. Let $x_0\in S$ be a point
such that $|v(x_0)|=\|v\|_{L^\infty(S)}$. By changing $v$ by $-v$ we can always assume that $v(x_0)>0$.
Then, from \eqref{eq:contraction} and \eqref{eq:power} we conclude that
$L^s_\f v(x_0)\geq 0$. The case $L^s_\f v(x_0)>0$ is excluded by hypothesis.
Thus, $e^{-tL_\f}v(x_0)=v(x_0)$ for every $t>0$, contradicting \eqref{eq:exponential decay semigroup}.
The proof of uniqueness is complete. \qed

\subsection{The extension problem}

Using the language of semigroups we can characterize $L^s_\f$ with the extension problem
established in \cite[Theorem~1.1,~Theorem~2.1]{Gale-Miana-Stinga}.
Observe that the main results in \cite{Gale-Miana-Stinga} apply in principle
to generators of bounded $C_0$-semigroups.
Nevertheless, it is easy to follow the proofs there and conclude that we can 
extend them to our present case.

Let $v\in\Dom_S(L_\f)$. For $x\in S$ and $z>0$ we define
\begin{equation}\label{def:V}
V(x,z):=\frac{(s^{2}z^{1/s})^s}{\Gamma(s)}\int_0^\infty e^{-(s^2z^{1/s})/t}e^{-tL_\f}v(x)\,\frac{dt}{t^{1+s}}.
\end{equation}
Then $V(\cdot,z)\in\Dom_S(L_\f)$ and, for each $z>0$,
$$
\|V(\cdot,z)\|_{L^\infty(S)}\leq\|v\|_{L^\infty(S)}.
$$
As a function of $z$, $V(x,z)$ is $C^\infty(0,\infty)$, for every $x\in S$.
Furthermore, $V$ is a classical solution to the extension problem
\begin{equation}\label{eq:extension nondivergence form}
\begin{cases}
-L_\f V+z^{2-1/s}V_{zz}=0,&\hbox{for}~x\in S,~z>0,\\
V(x,z)=0,&\hbox{for}~x\in\partial S,~z\geq0, \\
\displaystyle\lim_{z\to0^+}V(x,z)=v(x),&\hbox{uniformly in}~S.
\end{cases}
\end{equation}
Indeed, by usual elliptic regularity, $V\in C^{2,\alpha}_{\mathrm{loc}}(S\times(0,\infty))$ for every $0<\alpha<1$.
Moreover, $V(x, \cdot)\in C^1[0,\infty)$ and
\begin{equation}\label{Vz0=Lv}
-\lim_{z\to0^+}V_z(x,z)=d_sL_\f^sv(x),
\end{equation}
uniformly in $S$, where
\begin{equation}\label{def:ds}
d_s:=\frac{s^{2s}\Gamma(1-s)}{\Gamma(1+s)}>0.
\end{equation}
Hence, $V\in C^\infty((0,\infty);\Dom_S(L_\f))\cap \,C^1([0,\infty);C_0(\overline{S}))$
is a solution to the extension problem
\begin{equation}\label{EP:nonDiv:S}
\begin{cases}
\displaystyle-L_\f V+z^{2-1/s}V_{zz}=0,&\hbox{for}~x\in S,~z>0,\\
\displaystyle V(x,z)=0,&\hbox{for}~x\in\partial S,~z\geq0, \\
\displaystyle-\lim_{z\to0^+}V_z(x,z)=d_sL_\f^sv(x),&\hbox{uniformly in}~S.
\end{cases}
\end{equation}
For all the details see \cite{Gale-Miana-Stinga}.

It is clear that the even extension of $V(x,z)$ to $z\in\R$ given by
$$\widetilde{V}(x,z):=V(x,|z|),\quad\hbox{for}~(x,z) \in S \times \re,$$ 
satisfies
$$\widetilde{V} \in C^{2,\alpha}_{\mathrm{loc}}(\re \setminus\{0\};\Dom_S(L_\f))\cap \mathrm{Lip}(\re;C_0(\overline{S})),$$ 
for any $0<\alpha<1$, and solves
\begin{equation*}\label{PDE:Vtilde}
-L_\f\widetilde{V}+|z|^{2-1/s}\widetilde{V}_{zz}=0,\quad\hbox{in}~S \times (\re \setminus\{0\}).
\end{equation*}

\begin{rem}[Extension problem for negative powers]
Given a function $f\in C_0(\overline{S})$, let 
$v\in\Dom_S(L_\f)$ be the solution to $L_\f^sv=f$ in $S$.
The solution $V$ in \eqref{def:V} can also be written as
$$V(x,z)=\frac{1}{\Gamma(s)}\int_0^\infty
e^{-(s^2z^{1/s})/t}e^{-tL_\f}f(x)\,\frac{dt}{t^{1-s}}.$$
Then, for every $x\in S$ we readily get
$$V(x,0)=v(x)=L^{-s}_\f f(x).$$
For the details see \cite[Theorems~1.1,~Theorem~2.1]{Gale-Miana-Stinga},
also \cite{Caffarelli-Stinga, Gale-Miana-Stinga, Stinga Thesis, Stinga-Torrea CPDE}.
\end{rem}

\begin{rem}[Uniqueness]\label{rem:uniqueness}
By the weak maximum principle for elliptic equations, it is easy to see
that there is at most one solution to the extension problem \eqref{eq:extension nondivergence form}
such that
\begin{equation}\label{eq:uniformly to zero}
\lim_{z\to\infty}\|V(\cdot,z)\|_{L^\infty(S)}=0.
\end{equation}
Using the semigroup decay \eqref{eq:exponential decay semigroup} it is readily checked
that $V(x,z)$ as defined in \eqref{def:V} satisfies \eqref{eq:uniformly to zero}, so 
this is indeed the unique solution.
\end{rem}

\section{An explicit example of $L_\f^sv(x)$}\label{sec:example:Ls}

In this section we give two explicit examples
on the action of $L_\f^s$. Our examples are inspired by the identity 
\begin{equation*}\label{Lff=n}
L_\f(-\f)=\tr((D^2\f)^{-1} D^2\f)=n.
\end{equation*}

\begin{thm}\label{thm:Ls:example}
Given an arbitrary section $S:=S_\f(x_0, R)$ introduce the function
\begin{equation*}
v_\f(x):= R - (\f(x) - \f(x_0) - \langle \nabla \f(x_0), x - x_0 \rangle) \quad \hbox{for}~x \in \overline{S}.
\end{equation*}
Then, for every $0<s<1$,
\begin{equation}\label{Lvf:vfs}
L_\v^s(v_\f)(x)=n^s v_\f(x)^{1-s},\quad \forall x\in S.
\end{equation}
\end{thm}

\begin{proof} Notice that $v_\f \in C_0(\overline{S}) \cap C^2(S) \subset \Dom_{S}(L_\f)$ and that $D^2 v_\f = -D^2 \f$ in $S$. Also, from the definition of $S_\f(x_0, R)$ in \eqref{def:section}, we have $v_\f > 0$  in $S$. In order to show \eqref{Lvf:vfs} we first need to find the unique solution $V$ to the extension equation
\begin{equation}\label{eq:Vvg}
\begin{cases}
-L_\v V+z^{2-1/s}V_{zz}=0,&\hbox{for}~x\in S,~z>0,\\
V(x,z)=0,&\hbox{for}~x\in\partial S,~z\geq0,\\
V(x,0)=v_\f(x),&\hbox{for}~x\in S,
\end{cases}
\end{equation}
that satisfies \eqref{eq:uniformly to zero} in Remark \ref{rem:uniqueness}.
We do so by pursuing a solution  $V$ of the form
$$V(x,z)= v_\f(x) g(z),$$
where $g:[0,\infty)\to\rn$ has to be found. 
Notice that
$$L_\v V=- \tr((D^2\v(x))^{-1}D^2\v(x))g(z)=ng(z).$$
Hence, by using the equation $-L_\v V+z^{2-1/s}V_{zz}=0$, the function $g\geq0$ must be
a solution to
\begin{equation}\label{eq:ODE}
g''+\bigg(\frac{-n}{v_\f(x)}\bigg)z^{1/s-2}g=0,\qquad g(0)=1,
\end{equation}
that decays to zero as $z\to\infty$.
The equation \eqref{eq:ODE} is a Bessel equation and
in order to find its unique solution we follow the analysis in \cite[Section~3.1]{Stinga-Torrea CPDE},
see also \cite{Lebedev} for details about Bessel functions.
To simplify the notation, fix $x \in S$ and set
\begin{equation}\label{eq:alpha}
\alpha:=\frac{n}{v_\f(x)} \in (0, \infty).
\end{equation}
Thus, the general solution to \eqref{eq:ODE} is
$$g(z)=z^{1/2}\mathcal{Z}_s(\pm i2s\alpha^{1/2}z^{1/(2s)}),$$
where $\mathcal{Z}_\nu(r)$ denotes a general cylinder function, see \cite[p.~106]{Lebedev}.
By using the boundary condition $|g(z)|\leq C$ 
as $z\to\infty$ (see the asymptotic expansions of Bessel functions
in \cite[Section~3.1]{Stinga-Torrea CPDE}) we obtain the modified Bessel function of the second kind $\mathcal{K}_\nu$:
$$g(z)=Cz^{1/2}\frac{2i^{-s-1}}{\pi}\mathcal{K}_s(2s\alpha^{1/2}z^{1/(2s)}),$$
where $C$ is an arbitrary constant. Recall that $\mathcal{K}_\nu(r)\sim 2^{\nu-1}\Gamma(\nu)r^{-\nu}$ as $r\to0$. Hence,
$$
g(z)\sim C\frac{i^{-s-1}\Gamma(s)}{\pi s^s\alpha^{s/2}},\quad\hbox{as}~z\to0,
$$
and, in order to satisfy the initial condition $g(0)=1$, we impose
$$C:=\frac{\pi s^s\alpha^{s/2}}{i^{-s-1}\Gamma(s)}.$$
Therefore,
\begin{equation}\label{eq:g}
g(z)=\frac{2^{1-s}}{\Gamma(s)}(2s\alpha^{1/2}z^{1/(2s)})^s\mathcal{K}_s(2s\alpha^{1/2}z^{1/(2s)}),
\end{equation}
is the unique bounded solution to \eqref{eq:ODE}. It is clear that $g\geq0$.
Moreover, from the asymptotic behavior $\mathcal{K}_\nu(r)\sim(\pi/(2r))^{1/2}e^{-r}$, as $r\to\infty$, it follows that
$g(z)\to0$ exponentially, as $z\to\infty$. In conclusion,
\begin{equation}\label{eq:solucionV}
V(x,z)= v_\f(x) g(z),
\end{equation}
with $g(z)$ as in \eqref{eq:g} and $\alpha$ as in \eqref{eq:alpha}, is the unique positive
bounded solution to \eqref{eq:Vvg} that satisfies \eqref{eq:uniformly to zero}, see Remark \ref{rem:uniqueness}.
On the other hand, from \eqref{Vz0=Lv}
we know that if $V$ is the unique solution to the extension problem \eqref{eq:Vvg} then
\begin{equation}\label{eq:condicionNeumann}
-\frac{\Gamma(1+s)}{s^{2s}\Gamma(1-s)}V_z(x,0)=L^s_\v(v_\f)(x) \quad \forall x \in S.
\end{equation}
Again, to simplify the computation, let us put
$$
\beta=\beta(z):=2s\alpha^{1/2}z^{1/(2s)},
$$
so that
$$g(z)=\frac{2^{1-s}}{\Gamma(s)}\beta^s\mathcal{K}_s(\beta).$$
By using the chain rule, the properties of the derivatives of Bessel functions
and the fact that $\mathcal{K}_{-\nu}(r)=\mathcal{K}_{\nu}(r)$, we can compute
\begin{equation}\label{eq:computation}
\begin{aligned}
\frac{dg}{dz}(z) &= \frac{2^{1-s}}{\Gamma(s)}\frac{d}{d\beta}(\beta^s\mathcal{K}_s(\beta))\frac{d\beta}{dz} \\
&= -\frac{2^{1-s}}{\Gamma(s)}\beta^s\mathcal{K}_{s-1}(\beta)\alpha^{1/2}z^{1/(2s)-1} \\
&= -\frac{2s^s}{\Gamma(s)}\alpha^{(s+1)/2}z^{1/(2s)-1/2}\mathcal{K}_{1-s}(2s\alpha^{1/2}z^{1/(2s)}).
\end{aligned}
\end{equation}
Whence, from \eqref{eq:solucionV}, \eqref{eq:computation} and the asymptotic behavior of
$\mathcal{K}_\nu(r)$ as $r\to0$, we get
\begin{align*}
-\lim_{z\to0^+}V_z(x,0) &=  \frac{2s^s}{\Gamma(s)}v_\f(x)\alpha^{(s+1)/2}\lim_{z\to0^+}
z^{1/(2s)-1/2}\mathcal{K}_{1-s}(2s\alpha^{1/2}z^{1/(2s)}) \\
&= \frac{2s^s}{\Gamma(s)}v_\f(x)\alpha^{(s+1)/2}\lim_{z\to0^+}
z^{1/(2s)-1/2}\frac{2^{-s}\Gamma(1-s)}{(2s\alpha^{1/2}z^{1/(2s)})^{1-s}} \\
&= \frac{\Gamma(1-s)s^{2s}}{\Gamma(1+s)}v_\f(x)\alpha^s.
\end{align*}
Recalling the definition of $\alpha$ from \eqref{eq:alpha} and
the identity \eqref{eq:condicionNeumann}, we arrive at \eqref{Lvf:vfs}. \end{proof}

As a consequence of Theorem \ref{thm:Ls:example} we obtain the following result on the fractional Dirichlet Laplacian.

\begin{cor}
For every $0<s<1$ we have
$$(-\Delta_D)^s(1-|\cdot|^2)(x)=(2n)^s(1-|x|^2)^{1-s}, \quad \forall x\in B_1(0),$$
where $(-\Delta_D)^s$ is the fractional Dirichlet Laplacian in the unit ball $B_1(0) \subset \rn$.
\end{cor}

\begin{proof}
Use Theorem \ref{thm:Ls:example} with  $\varphi(x)\equiv\varphi_2(x):=|x|^2$ for every $x \in \rn$ and notice that $B_1(0) = S_{\f_2}(0,1)$ and $D^2 {\f_2} = 2 I$, which gives $L_{\f_2} = -\frac{1}{2} \Delta$. 
\end{proof}

\section{Divergence form: fractional powers $\L_\f^s$ and extension problem}\label{Section:extension divergence}

This section is devoted to the definition of the fractional powers of
the divergence form operator $\L_\f$ subject to the homogeneous Dirichlet boundary condition.
As in Section \ref{Section:extension nondivergence}, fix any section $S:=S_\f(x_0, R)$. For $f\in L^2(S,d\mu_\f)$,
consider the following Dirichlet problem for $\L_\f$:
\begin{equation}\label{eq:problem2}
\begin{cases}
\L_\varphi u\equiv-\dive(A_\f(x)\nabla u)=\mu_\varphi f,&\hbox{in}~S,\\
u=0,&\hbox{on}~\partial S.
\end{cases}
\end{equation}
Observe that the right hand side $f$ in \eqref{eq:problem2} appears multiplied by the Monge--Amp\`ere measure $\mu_\f$.
This can always be assumed by considering $f/\mu_\f$.

\subsection{The fractional nonlocal operator $\L^s_\f$}

The fractional powers $\L_\v^s$, $0<s<1$, will be defined by using the Dirichlet
eigenfunctions and eigenvalues along the lines of
\cite{Caffarelli-Stinga, Stinga Thesis, Stinga-Torrea CPDE, Stinga-Zhang}.

Let  $W^{1,2}_{0,\v}(S)$ denote the completion of $C^1_c(S)$ with respect to the norm
$$\|u\|_{W^{1,2}_{0,\v}(S)}^2:=\|u\|_{L^2(S,d\mu_\v)}^2+\|\nabla^\v u\|_{L^2(S,d\mu_\v)}^2.$$
Here $\nabla^\f$ stands for the Monge--Amp\`ere gradient, which is defined as
\begin{equation*}\label{def:grad:phi}
\nabla^\varphi u:=(D^2\varphi)^{-1/2}\nabla u.
\end{equation*}
By the Sobolev inequality for the Monge--Amp\`ere quasi-metric structure,
see \cite[Theorem 1]{MaMRL13}, an equivalent norm in
$W^{1,2}_{0,\v}(S)$ is $\|\nabla^\v u\|_{L^2(S,d\mu_\v)}$.

A weak solution $u$ to \eqref{eq:problem2} is a function $u\in W^{1,2}_{0,\varphi}(S)$ such that
$$
\int_S\langle\nabla^\varphi u, \nabla^\varphi h\rangle
\,d\mu_\varphi=\int_Sfh\,d\mu_\varphi, \quad \hbox{for every}~h\in W^{1,2}_{0,\v}(S).
$$
Notice that the matrix of coefficients
$A_\varphi(x)$ is symmetric and uniformly elliptic in the compact set $\overline{S}$
and that $L^2(S,d\mu_\varphi)$ is isometrically embedded in $L^2(S,dx)$.
Therefore, by standard techniques (see for example \cite[Chapter~6]{Evans} or \cite[Section~8.12]{Gilbarg-Trudinger}),
there exist a sequence of eigenvalues
$0<\lambda_1<\lambda_2\leq\lambda_3\leq\cdots\leq\lambda_k\nearrow\infty$ and
a corresponding family of eigenfunctions $\{e_k\}_{k\geq1}\subset W^{1,2}_{0,\varphi}(S)$ such that
\begin{equation}\label{def:ek:lk}
\begin{cases}
\L_\varphi e_k=\mu_\varphi\lambda_ke_k,&\hbox{in}~S,\\
e_k=0,&\hbox{on}~\partial S,
\end{cases}
\end{equation}
in the weak sense. In other words, for every $h\in W^{1,2}_{0,\f}(S)$ and every $k \in \na$,
\begin{equation}\label{eq:eigenvalues weak}
\int_S\langle\nabla^\varphi e_k, \nabla^\varphi h\rangle\,d\mu_\varphi=\lambda_k\int_Se_kh\,d\mu_\varphi.
\end{equation}
Moreover, $\{e_k\}_{k\in \na}$ forms an orthonormal basis of $L^2(S,d\mu_\varphi)$.

For $s\geq0$, we consider the Hilbert space
\begin{equation}\label{def:HsS}
\H^s_\v(S) := \Dom_S(\L_\v^s)
:=\Big\{u=\sum_{k=1}^\infty u_ke_k\in L^2(S,d\mu_\v):\sum_{k=1}^\infty\lambda_k^su_k^2<\infty\Big\},
\end{equation}
endowed with the inner product
$$\langle u,h\rangle_{\H^s_\f(S)}:=\sum_{k=1}^\infty\lambda_k^su_kh_k,\quad \hbox{for}~u,h\in\H_\f^s(S),$$
where $h=\sum_{k=1}^\infty h_ke_k$.
Observe that $\H^0_\f(S)=L^2(S,d\mu_\f)$. 
From \eqref{eq:eigenvalues weak} it is readily verified that $\H^1_\f(S)=W^{1,2}_{0,\f}(S)$ as Hilbert spaces and
\begin{equation}\label{eq:H1W1}
\int_S \langle \nabla^\varphi u, \nabla^\varphi h \rangle \,d\mu_\varphi
=\sum_{k=1}^\infty\lambda_ku_kh_k,\quad \hbox{for any}~u,h\in\H^1_\f(S).
\end{equation}
We read the right-hand side in \eqref{eq:H1W1} as the definition of $\L_\f u$ for $u\in\H^1_\v(S)$. We are now in position to define the fractional power $\L_\v^s$ in $\H^s_\v(S)$.

\begin{defn} 
Let $0<s<1$. The fractional operator $\L_\f^su$ is defined for any $u\in\H^s_\v(S)$
as the unique element $\L_\v^s u$ in the dual space $\H^s_\v(S)'$ acting as
\begin{equation}\label{eq:L fancy fraccionario}
(\L_\varphi^s u)(h)=\sum_{k=1}^\infty\lambda_k^su_kh_k,\quad
\hbox{for every}~h=\sum_{k=1}^\infty h_ke_k\in\H^s_\v(S).
\end{equation}
\end{defn}

\subsection{Proof of Theorem \ref{thm:existence}(ii)}\label{subsec:proof:exist:LL} 

Given $F=\sum_{k=1}^\infty F_ke_k$ in $\H_\v^s(S)'$, the unique solution $u\in \H_\v^s(S)$ is given by
$u=\sum_{k=1}^\infty \lambda_k^{-s}F_ke_k\in \H_\f^s(S)$.
In particular, if $F\in L^2(S,d\mu_\f)$ then $u\in\H^{2s}_\f(S)$.\qed

\subsection{The extension problem}

Let us introduce 
\begin{equation*}\label{def:a}
a:=1-2s\in(-1,1),
\end{equation*}
and, for $x\in S$, 
\begin{equation}\label{eq:matriz B}
B_\varphi(x):=
\begin{pmatrix}
A_\f(x) & 0 \\
0 & \muf(x)
\end{pmatrix}
\in \R^{n+1}\times\R^{n+1}.
\end{equation}
By reasoning as in \cite{Stinga Thesis, Stinga-Torrea CPDE}, see also \cite{Caffarelli-Stinga, Gale-Miana-Stinga}, the extension
problem characterization of \eqref{eq:L fancy fraccionario} can now be obtained as follows. Let $u\in\H^s_\f(S)$.
We say that a function $U=U(x,y)$,
defined for $x\in S$ and $y\geq0$, is a weak solution to the extension problem
\begin{equation}\label{eq:extension divergence form}
\begin{cases}
\displaystyle\dive_{x,y}(y^a B_\f(x)\nabla_{x,y}U)=0,&\hbox{for}~x\in S,~y>0,\\
\displaystyle U(x,y)=0,&\hbox{for}~x\in\partial S,~y\geq0, \\
\displaystyle U(x,0)=u(x),&\hbox{for}~x\in S,
\end{cases}
\end{equation}
if $U$, $\nabla^\f U$ and $U_y$ belong to the weighted
space  $L^2(S\times(0,\infty),y^ad\mu_\f dy)$,
$U=0$ on $\partial S\times(0,\infty)$ in the sense of traces,
$U(x,y)\to u(x)$ as $y\to0^+$ in $L^2(S,d\mu_\f)$,
and for every test function $W=W(x,y)$
such that $W(x,0)=0$ in $S$, we have
\begin{align*}
\int_0^\infty\int_S&y^a\langle B_\f(x)  \nabla_{x,y}U, \nabla_{x,y}W \rangle\,dx\,dy \\
&=\int_0^\infty\int_Sy^a \langle \nabla^\f U, \nabla^\f W \rangle \,d\mu_\f\,dy+\int_0^\infty\int_Sy^aU_yW_y\,d\mu_\f \,dy=0.
\end{align*}
By proceeding as in \cite[Section~3.3.1]{Stinga Thesis} or \cite[Section~3.1]{Stinga-Torrea CPDE},
see also \cite[Section~2.3]{Caffarelli-Stinga}, the unique weak solution
$U$ that weakly vanishes as $y\to\infty$ can be written using the
Fourier coefficients $u_k$ of $u$ and the eigenfunctions $e_k$ as
$$U(x,y)=\sum_{k=1}^\infty c_k(y)u_ke_k(x),$$
where the coefficients $c_k(y)$ are given by
$$c_k(y)=\frac{2^{1-s}}{\Gamma(s)}(\lambda_k^{1/2}y)^s\mathcal{K}_s(\lambda_k^{1/2}y),
\quad\hbox{for}~y>0~\hbox{and}~k\geq1.$$
Here $\mathcal{K}_s$ is the modified Bessel function of the second kind and parameter $s$.
Moreover,
\begin{equation}\label{eq:normal derivative U}
-\lim_{y\to0^+}y^aU_y=c_s\L_\f^su,\quad\hbox{in}~\H^s_\f(S)',
\end{equation}
where
$$
c_s:=\frac{\Gamma(1-s)}{4^{s-1/2}\Gamma(s)}.
$$
That is, if $W(x,y)$ is a test function then \eqref{eq:normal derivative U} reads
$$\int_0^\infty\int_Sy^a \langle \nabla^\f U, \nabla^\f W \rangle \,d\mu_\f\,dy+\int_0^\infty\int_Sy^aU_yW_y \, d\mu_\f \,dy
=c_s(\L_\v^su)(W(\cdot,0)).$$
By using $U$ as a test function we obtain the energy identity
\begin{equation}\label{L2energy}
\int_0^\infty\int_Sy^a|\nabla^\f U|^2\,d\mu_\f\,dy+\int_0^\infty\int_Sy^a|U_y|^2\,d\mu_\f \,dy
=c_s\int_S|\L_\v^{s/2}u|^2\,d\muf,
\end{equation}
where we used that $\L^{s/2}_\f u\in L^2(S,d\mu_\f)$.

\begin{rem}\label{rem:semigroup definition}
As in \cite{Caffarelli-Stinga} we could have used the semigroup generated by $\L_\f$
to obtain an equivalent expression for $\L^s_\f u$ which explicitly shows that the fractional
operator $\L^s_\f$ is a nonlocal integro-differential operator in divergence form.
It is also possible to write the solution $U$ above by using the semigroup generated by
$\L_\v$ (see \cite{Caffarelli-Stinga, Stinga Thesis, Stinga-Torrea CPDE}). Instead, by means of a change of variables,  in Section \ref{sec:Ls=LLs} we will directly relate the solution $U$ of the extension problem in divergence form \eqref{eq:extension divergence form} to the solution of the extension problem \emph{in nondivergence form} \eqref{eq:extension nondivergence form}. 
\end{rem}

\section{nondivergence form meets divergence form: proof of Theorem \ref{thm:L=LL}}\label{sec:Ls=LLs}

In this section we establish the connection between the nondivergence form and divergence
form extension problems \eqref{eq:extension nondivergence form} and \eqref{eq:extension divergence form},
which will ultimately leads us to the proof of Theorem \ref{thm:L=LL}. Let us start with the following proposition. 

\begin{prop}\label{prop:DL:in:DLL} 
For every section $S:=S_\f(x_0, R)$ the following inclusion holds true
\begin{equation}\label{DL:in:DLL}
\Dom_S(L_\f^s) \subset \Dom_S(\L_\f^s).
\end{equation}
\end{prop}

\begin{proof}
Recall from Definition \ref{definition:fractional nondivergence}
that $\Dom_S(L_\f^s) = \Dom_S(L_\f)$ as previously defined in \eqref{def:DomL}.
On the other hand, $\Dom_S(\L_\f^s) = \H^s_\v(S)$, see \eqref{def:HsS}.
Let us remark that, since the eigenvalues $\{\lambda_k\}_{k\geq1}$ 
from \eqref{def:ek:lk} increase towards $+\infty$, we have $\lambda_k^s < \lambda_k$
for all sufficiently large $k$, and then $\H^1_\f(S) \subset \H^s_\f(S)$. Also, the fact that the Hessian $D^2\f$ is
a positive definite matrix with entries in $C(\rn)$ implies that $L^2(S,dx)= L^2(S,\dmuf)$ and
$W^{1,2}_{0, \f}(S) = W^{1,2}_{0}(S)$ as Hilbert spaces.
Here $W^{1,2}_{0}(S)$ denotes the usual Sobolev space with respect to Lebesgue measure. 

Now, given  $v \in \Dom_S(L_\f)$ we have $v \in C_0(\overline{S})$ as well as $L_\f v\in C(\overline{S}) \subset L^2(\overline{S})$, and by the $L^2(S)$-solvability theorem for the Dirichlet problem  (see, for instance, Theorem 9.15 on \cite[p.~241]{Gilbarg-Trudinger} and notice that $\partial S \in C^2$), we obtain $v \in W^{2,2}(S)$, which, combined with $v \in C_0(\overline{S})$, gives $v \in W^{1,2}_{0}(S)$. Hence,  
$$
v \in W^{1,2}_{0}(S)= W^{1,2}_{0, \f}(S) = \H^1_\f(S) \subset \H^s_\f(S),
$$
which proves \eqref{DL:in:DLL}.
\end{proof}

\subsection{Proof of Theorem \ref{thm:L=LL}}

The connection between the fractional powers $L_\f^s$ and $\L_\f^s$ will materialize through the change of variables
\begin{equation}\label{change:yz}
z=(y/(2s))^{2s},~z>0\quad\longleftrightarrow\quad y=(2s)z^{1/(2s)},~y>0,
\end{equation}
see \cite{Caffarelli-Silvestre CPDE}. Define
\begin{equation}\label{eq:U and V}
U(x,y):=V(x,z),
\end{equation}
for $x\in S$ and $y>0$, where $V$ is the unique solution to
the extension equation satisfying \eqref{eq:uniformly to zero}, see Section \ref{Section:extension nondivergence}.  Then, 
\begin{equation}\label{Uy:Vz}
U_y=V_zz_y=(y/(2s))^{2s-1}V_z,
\end{equation}
and
\begin{equation*}\label{Uyy:Vzz}
U_{yy}=(y/(2s))^{4s-2}V_{zz}+\frac{2s-1}{2s}(y/(2s))^{2s-2}V_z.
\end{equation*}
Therefore,
\begin{equation}\label{eq:casi casi}
\tfrac{a}{y}U_y+U_{yy}=(y/(2s))^{4s-2}V_{zz}=z^{2-1/s}V_{zz}.
\end{equation}
Also, from \eqref{L=calL},
\begin{equation}\label{LV=LU}
L^\f V=\L_\f V=\L_\f U, \quad \text{in } S \times (0, \infty),
\end{equation}
where in the second identity we noticed that $\L_\f$ acts only in the variable $x\in S$ for each fixed $z>0$ and $y>0$.
Therefore, from \eqref{LV=LU} and \eqref{eq:casi casi}, since $V$ is the solution to the extension equation
\eqref{eq:extension nondivergence form},
\begin{align*}
0=-L^\f V+\mu_\f z^{2-1/s}V_{zz}&=-\L_\f U+\mu_\f\big(\tfrac{a}{y}U_y+U_{yy}\big) \\
&= y^{-a}\dive_{x,y} (y^a B_\varphi(x) \nabla_{x,y}U),
\end{align*}
where $B_\f(x)$ is as in \eqref{eq:matriz B}. Therefore, $U$ defined by \eqref{eq:U and V}
is a solution to \eqref{eq:extension divergence form}
with $U(x,0)=V(x,0)=v(x)$ for $v \in \Dom_S(L_\f)$.
Moreover, by \eqref{eq:uniformly to zero}, we see that $U(\cdot,y)\to 0$, as $y\to\infty$, weakly in $L^2(S,d\mu_\f)$.
Hence $U$ in \eqref{eq:U and V} is the unique solution to \eqref{eq:extension divergence form}.

From \eqref{eq:normal derivative U}, for every $x\in S$,
$$-\lim_{y\to0^+}y^aU_y(x,y)=c_s\L_\f^sv(x).$$
On the other hand, it is readily verified that
$$-y^{1-2s}U_y=-\frac{1}{(2s)^{2s-1}}V_z.$$
Thus,
\begin{equation*}
c_s\L_\f^s v(x) = -\lim_{y\to0^+}y^aU_y(x,y) = -\frac{1}{(2s)^{2s-1}}\lim_{z\to0^+}V_z(x,z) = \frac{d_s}{(2s)^{2s-1}}L_\f^sv(x).
\end{equation*}
Hence, as $c_s=d_s/(2s)^{2s-1}$, for $v \in \Dom_S(L_\f)$ we get 
\begin{equation}\label{Lv=LLv}
L^s_\f v=\L_\f^s v.
\end{equation}
Therefore, $L^s_\f v(x)=\L^s_\f v(x)$ and both the divergence and nondivergence structures occur simultaneously. 
\qed

\begin{rem} Notice that, in order to keep the equality in \eqref{Lv=LLv} consistent with \eqref{L=calL}, the righthand sides of $L^s_\f v$ and $\L_\f^s v$ in \eqref{eq:vfuf} must differ by $\muf$, as stemming from the equations \eqref{eq:problem1} and \eqref{eq:problem2}. 
\end{rem}

\begin{rem}\label{rmk:finite:E}
The following finite-energy estimate will play a key role in Section \ref{sec:S(Q)}:
\begin{equation}\label{E:E:V}
 \int_{S}\int_0^{\infty} |\nf V(x,z)|^2 z^{1/s-2} \dz \dmuf(x) + \int_{S}\int_{0}^{\infty} V_z(x,z)^2 \dz \dmuf(x) < \infty.
\end{equation}
We prove \eqref{E:E:V} by using the change of variables \eqref{change:yz}. Indeed, 
\begin{align*}
\int_{S} \int_0^{\infty} &|\nf V(x,z)|^2 z^{1/s-2} \dz \dmuf(x) \\
& = \int_{S} \int_0^{\infty} |\nf U(x,y)|^2 \left(\frac{y}{2s}\right)^{2s \left(1/s-2\right)}  \left(\frac{y}{2s}\right)^{2s-1} \dy\dmuf(x) \\
& = (2s)^{2s -1} \int_{S_\f} \int_0^{\infty} |\nf U(x,y)|^2 y^{a} \dy \dmuf(x).
\end{align*}
On the other hand, from \eqref{Uy:Vz}, 
\begin{align*}
\int_{S}  \int_{0}^{\infty} V_z(x,z)^2 \dz \dmuf(x)& = \int_{S}  \int_{0}^{\infty} U_y(x,y)^2  \left(\frac{y}{2s}\right)^{1-2s} \dy \dmuf(x)\\
& = (2s)^{2s -1}  \int_{S} \int_{0}^{\infty} U_y(x,y)^2 y^a \dy \dmuf(x).
\end{align*}
Therefore, by the energy identity \eqref{L2energy},
\begin{align*}
\int_{S} \int_0^{\infty}& |\nf V(x,z)|^2 z^{1/s-2} \dz \dmuf(x) + \int_{S}  \int_{0}^{\infty} V_z(x,z)^2 \dz \dmuf(x)\\
& =  (2s)^{2s -1}  \bigg[\int_{S} \int_0^{\infty} |\nf U(x,y)|^2 y^{a} \dy \dmuf(x)+ \int_{S} \int_{0}^{\infty} U_y(x,y)^2 y^a \dy \dmuf(x) \bigg]\\
 & = (2s)^{2s -1}  c_s\int_{S} |\L_\v^{s/2}u(x)|^2 \dmuf(x) < \infty. \qed
\end{align*}
\end{rem}

\section{Notation and Monge--Amp\`ere background}\label{sec:background}

Throughout the article, the function $\phi$ will denote a generic convex function which is used as a placeholder for the functions $\f$, $\F$, and $h_s$ to be introduced in Section \ref{sec:def:Phi}. Let also $N$ denote a generic dimension that will take the values $n$, $n+1$, or $1$. 

For a strictly convex function $\phi \in C^1(\rN)$ (strictly convex in the sense that its graph contains no line segments), its associated Monge--Amp\`ere measure $\mu_\phi$ acts on a Borel set $E \subset \rN$ as
\begin{equation*}
\mu_\phi(E):= |\nabla \phi(E)|,
\end{equation*}
where $|F|$ denotes the Lebesgue measure of a subset $F \subset \rN$. Given $x \in \rN$ and $R > 0$, its Monge--Amp\`ere section $S_\phi(x,R)$ is defined as the open, convex set
\begin{equation*}\label{def:SxR}
S_\phi(x, R) := \{ y \in \rN: \delta_\phi(x,y) < R\}
\end{equation*}
where
\begin{equation}\label{def:delta:f}
\delta_\phi(x,y):= \phi(y) - \phi(x) - \langle \nabla \phi(x), y-x \rangle \quad \forall x, y \in \rN.
\end{equation}
We write $\mu_\phi \in \dcc_\phi$ if there exists a constant $C_d \geq 1$ such that
\begin{equation}\label{def:DC}
\mu_\phi(S_\phi(x,t)) \leq C_d \,  \mu_\phi(\tfrac{1}{2} S_\phi(x,t)) \quad \forall x \in \rN, \forall t > 0,
\end{equation}
where, for a convex set $S$, $\tfrac{1}{2} S$ denotes its $\tfrac{1}{2}$-contraction with respect to its center of mass (with the computation of the center of mass based on the Lebesgue measure). The condition $\mu_\phi \in \dcc_\phi$ is equivalent to the structure of space of homogeneous type for the triple $(\rN, \mu_\phi, \delta_\phi)$ (see \cite{MaldonadoCVPDE14} and references therein), which we understand as the minimal structure to carry out real analysis. It is in this sense that we refer to $\mu_\phi \in \dcc_\phi$ as a minimal geometric condition. 

If $\phi :\rN \to \re$ is twice differentiable at point $x_0 \in \rN$ with $\mu_\phi(x_0)= \det D^2 \phi(x_0) > 0$, then the matrix of cofactors of $D^2\phi$ at $x_0$ is given by
\begin{equation*}
A_\phi(x_0):= D^2\phi(x_0)^{-1} \mu_\phi(x_0).
\end{equation*}
If $\phi$ is three times differentiable at a point $x_0 \in \rN$ with $\mu_\phi(x_0) > 0$, then $A_\phi(x_0)$ has divergence-free columns (see, for instance, \cite[p.~462]{Evans}). Then, given $h : \rN \to \re$ twice differentiable at $x_0$ we have
\begin{equation*}\label{NL}
\dive(A_\phi(x_0) \nabla h(x_0)) = \tr(A_\phi(x_0) D^2 h(x_0)). 
\end{equation*}

\subsection{Convex conjugates}

Suppose that $\phi \in C^2(\rN)$ with $D^2 \phi > 0$. If $\mu_\phi \in \dcc_\phi$ then $\nabla \phi :\rN \to \rN$ is a continuously differentiable homeomorphism and the convex conjugate of $\phi$, denoted by $\psi : \rN \to \re$, satisfies
\begin{equation}\label{psi:phi:inv}
\nabla \psi(\nabla \phi) = \nabla \phi(\nabla \psi) = id : \rN \to \rN.
\end{equation}
In particular, we have
\begin{equation}\label{psi:C2}
\psi \in C^2(\rN) \quad \text{and} \quad D^2 \psi > 0. 
\end{equation}
In addition,  $\mu_\phi \in \dcc_\phi$ implies $\mu_\psi \in \dcc_\psi$ (with doubling constants depending on the ones for $\mu_\phi$ and dimension $N$) and the Monge--Amp\`ere sections of $\phi$ and $\psi$ are related as follows 
\begin{equation}\label{sec:phi:phi}
S_\phi(x, \kappa_1 R) \subset \nabla \psi(S_\psi(\nabla \phi(x), R)) \subset S_\phi(x, K_1 R)  \quad \forall x \in \rN, R > 0,
\end{equation}
for constants $0 < \kappa_1 < 1 < K_1 < \infty$ depending only on $C_d$ in \eqref{def:DC} and dimension $N$. See \cite[Section 5]{FM04} for these and related results.

\subsection{Poincar\'e inequalities}

Associated to a strictly convex $\phi \in C^2(\rN)$ with $D^2 \phi > 0$ and $\mu_\phi \in \dcc_\phi$, Poincar\'e inequalities with respect to its Monge--Amp\`ere sections and gradient $\nabla^\phi:=(D^2\phi)^{-1/2}\nabla$ have been proved in \cite[Theorem 1.3]{MaldonadoCVPDE14}. Namely, there exists a constant $C_P >0$, depending only on the $\dcc_\phi$ constant and $N$, such that for every section $S_\phi:=S_\phi(x_0,R)$ and every $u \in C^1(S_\phi)$ we have 
\begin{equation}\label{Poincare:phi}
\frac{1}{|S_\phi|}\int_{S_\phi} |u(x) - u_{S_\phi}| \dx  \leq C_{P} \, R^\frac{1}{2} \left(\frac{1}{|S_\phi|} \int_{S_\phi} |\nabla^\phi u(x)|^2 \dx \right)^\frac{1}{2},
\end{equation}
where $\displaystyle u_{S_\phi}:= \frac{1}{|S_\phi|}\int_{S_\phi} u(x) \dx$ and $|\nabla^\phi u(x)|^2  = \langle D^2\phi(x)^{-1} \nabla u(x), \nabla u(x) \rangle$. 

From now on, for a Borel measure $\mu$, which will be either a Monge--Amp\`ere measure $\mu_\phi$ or the Lebesgue measure, and a measurable set $E \subset \rN$ we put
\begin{equation*}
\fint_E f(x) \, d\mu(x) := \frac{1}{\mu(E)}\int_E f(x) \, d\mu(x).
\end{equation*}

\subsection{The $L^2(S, d\mu_\phi)$-energy of the quasi-distance $\delta_\phi$}

Here we record the following consequence of Lemma 3.1 from \cite{MaldonadoW1p}: given any strictly convex function $\phi \in C^3(\rN)$ (no doubling assumptions required) and any section $S:=S_\phi(x_0, R)$, we have 
\begin{equation}\label{Dxx0:phi}
\int\limits_{S}  \langle D^2 \phi(x)^{-1} (\nabla \phi(x) - \nabla \phi(x_0)), \nabla \phi(x) - \nabla \phi(x_0) \rangle \,d\mu_\phi(x) \leq n R \mu_\phi(S_\phi(x_0, R)).
\end{equation}
By recalling the definition of $\delta_\phi$ from \eqref{def:delta:f}, for each fixed $x_0 \in \rN$, we have
$$
\nabla \delta_\phi(x_0, x) =  \nabla \phi(x) - \nabla \phi(x_0), 
$$ 
and then
$$
|\nabla^\phi\delta_\phi(x_0, x)|^2 =\langle D^2 \phi(x)^{-1} (\nabla \phi(x) - \nabla \phi(x_0)), \nabla \phi(x) - \nabla \phi(x_0) \rangle,
$$ 
which makes \eqref{Dxx0:phi} an estimate on the $L^2(S, d\mu_\phi)$-energy, with respect to the Monge--Amp\`ere gradient $\nabla^\phi$, of the mapping $x \mapsto \delta_\phi(x_0, x)$. 

\section{The function $\F$}\label{sec:def:Phi}

Henceforth, fix $\f \in C^3(\rn)$ with $D^2 \f > 0$ in $\rn$ and $\muf \in \dcc_\f$. Given $0<s<1$ introduce
 \begin{equation*}\label{def:hs}
 h_s(z):= \frac{s^2   }{(1-s)} |z|^{1/s},\quad \forall z \in \re,
 \end{equation*}
and set
\begin{equation}\label{def:Phi:s}
\F(x,z):= \f(x) + h_s(z), \quad \forall (x,z) \in \rn \times \re.
\end{equation}
Observe that both $h_s$ and $\Phi$ are strictly convex (in the sense that their graphs do not contain line segments), continuously differentiable functions. From \eqref{def:Phi:s}, for every $(x,z) \in \rn \times (\re \setminus \{0\})$ we have 
\begin{equation}\label{def:muF}
\muF(x,z) = \muf(x) \mu_{h_s}(z)=  \muf(x) {h_s''}(z)= \muf(x) |z|^{1/s-2}.
\end{equation}

\subsection{The function $\F$ and the (DC) doubling property}

By defining \eqref{def:delta:f} for $\F$ as in \eqref{def:Phi:s}, for $X=(x,z), X_0=(x_0, z_0) \in \re^{n+1}$, we have
\begin{equation}\label{exp:delta:F}
\begin{aligned}
\delta_\F&(X_0, X)= \delta_\f(x_0, x) + \delta_{h_s}(z_0, z)\\
& = \f(x) - \f(x_0) - \langle \nabla \f(x_0), x-x_0 \rangle + h_s(z) - h_s(z_0) - h_s'(z_0)(z -z_0).
\end{aligned}
\end{equation}

Notice that, since $0 < s <1$, the function $h_s''(z) =  |z|^{1/s-2}$ is a  Muckenhoupt $A_p$-weight on the real line for some $1 < p < \infty$ if and only if $1/s < p+1$ and $h_s'' \in A_1$ if and only if $1/s \leq 2$ (see Example 9.1.7 on \cite[p.~286]{Gr2}). Hence, $h_s'' \in A_\infty$ for every $0<s<1$, which makes it a doubling weight on the real line (equivalently, $h_s \in \dcc_{h_s}$, since in dimension 1 the $\dcc$ doubling property coincides with the usual doubling property) whose doubling constant depends only on $s$. 
In particular, there exists $K_s \geq 1$, depending only on $0<s<1$, such that 
\begin{equation}\label{def:Ks}
\delta_{h_s}(z, z') \leq K_s \left(\min\{\delta_{h_s}(z,z''), \delta_{h_s}(z'',z) \} + \min\{ \delta_{h_s}(z',z''), \delta_{h_s}(z'',z') \} \right) 
\end{equation}
for every $z, z', z'' \in \re$. 

Now, since $\muf \in \dcc_\f$ and $h_s \in \dcc_{h_s}$, from \cite[Lemma 6]{FM09} it follows that  $\Phi$, being the tensor sum of $\f$ and $h_s$, satisfies $\F \in \dcc_\Phi$ with constants depending only on the $\dcc$ constants for $\muf$, dimension $n$, and $s$.  In addition, the condition $\muF \in \dcc_\F$ is quantitatively equivalent to the existence of $K \geq 1$ such that
\begin{equation}\label{def:K}
\delta_\F(X,Y) \leq K \left(\min\{\delta_\F(Z,X), \delta_\F(X,Z) \} + \min\{ \delta_\F(Z,Y), \delta_\F(Y,Z) \} \right) 
\end{equation}
for every $X, Y, Z \in \re^{n+1}$.

By \cite[Lemma 6]{FM09} the sections of $\F$ are related to the ones of $\f$ and $h_s$ by
\begin{equation}\label{S:vs:SxS}
S_\Phi((x_0,z_0), R) \subset S_\f(x_0, R) \times S_{h_s}(z_0, R) \subset S_\Phi((x_0, z_0), 2R) 
\end{equation}
for every $(x_0,z_0) \in \rn \times \re $ and $R > 0$. 

We recall that constants depending only on the $\dcc$ constants for $\muf$ in \eqref{def:DC}, $0<s<1$ and dimension $n$ will be called \emph{geometric constants}. 

By \cite[Corollary 3.3.2]{guti}, the condition $\muF \in \dcc_\F$ implies the following doubling property for $\muF$: there exists a geometric constant $K_d > 1$ such that
\begin{equation}\label{def:Kd}
\muF(S_\F(X, 2R)) \leq K_d \, \muF(S_\F(X, R))\quad \forall X \in \re^{n+1}, R > 0.
\end{equation}
Iterations of \eqref{def:Kd} yield
\begin{equation}\label{doubling:r:R}
\muF(S_\F(X, R)) \leq K_d \left( \frac{R}{r}\right)^\nu \muF(S_\F(X, r)) \quad \forall X \in \re^{n+1}, 0 < r < R,
\end{equation}
where $\nu: =\log_2 K_d$. Also, by $\muF \in \dcc_\F$ (as well as the hypotheses $\F \in C^1(\re^{n+1})$ and its strict convexity), there exists a geometric constant $K_3 > 1$ such that for every section $S:=S_\F(X, R)$ we have
\begin{equation}\label{def:K3}
\muF(S) |S| \leq K_3 \,R^{n+1},
\end{equation}
see for instance \cite[Theorem 1]{FM04}. 

\subsection{The matrix of cofactors of $D^2 \Phi$}

From \eqref{def:Phi:s}, for every $(x,z) \in \rn \times (\re \setminus \{0\})$
\begin{equation}\label{def:D2Finv}
D^2 \Phi(x,z)^{-1}=
\begin{pmatrix}
D^2\v(x)^{-1} & 0 \\
0 & |z|^{2-1/s}
\end{pmatrix}
\in \R^{n+1}\times\R^{n+1},
\end{equation}
where we have excluded the value $z=0$ to avoid the singularities of $|z|^{1/s-2}$ or $|z|^{2-1/s}$. 

From \eqref{def:muF} and \eqref{def:D2Finv},  for every $(x,z) \in \rn \times (\re \setminus \{0\})$ the matrix of cofactors of $D^2 \Phi(x,z)$ equals
\begin{equation*}\label{def:AFs}
A_\Phi(x,z):=
\begin{pmatrix}
A_\f(x) |z|^{1/s -2}& 0 \\
0 & \muf(x)
\end{pmatrix}
\in \R^{n+1}\times\R^{n+1}.
\end{equation*}
Important features are that $A_\Phi(x,z) \in L^1_{\mathrm{loc}}(\rn \times \re)$ and $A_\Phi(x,z)$ is differentiable for (Lebesgue) a.e. $(x,z) \in \rn \times \re$. Notice that the first $n$ columns of $A_\Phi$ are differentiable with respect to $x$ and the last column is differentiable with respect to $z$. Also, since the columns of $A_\f$ are divergence free, so are the columns of $A_\Phi$. 

If $H : \re^{n+1} \to \re$ is differentiable at a point $X=(x,z) \in  \rn \times (\re \setminus \{0\})$, the Monge--Amp\`ere gradient of $H$ at $X$ is then given by
\begin{equation*}
\nabla^\F H(X) = D^2 \F(X)^{-1/2} \nabla H(X) = (D^2 \f(x)^{-1/2} \nabla_x H(x, z),   |z|^{1-\frac{1}{2s}} H_z(x,z)) \in \re^{n+1},
\end{equation*}
which implies
\begin{align}\label{exp:nablaF:H:2}
|\nabla^\F H(X)|^2 &= \langle  D^2 \f(x)^{-1} \nabla_x H(x, z), \nabla_x H(x, z) \rangle +   |z|^{2-1/s} H_z(x,z)^2\\\nonumber
& = |\nabla^\f H(x,z)|^2 +   |z|^{2-1/s} H_z(x,z)^2.
\end{align}

\subsection{The $L^2(S_\F, d\mu_\F)$-energy of the quasi-distance $\delta_\F$}

For $\Phi$ as in \eqref{def:Phi:s}, we will next prove the following counterpart to \eqref{Dxx0:phi}: For every section $S_\F(X_0, R)$ it holds true that
\begin{align}\label{L2:Energy:delta:Phi}
&\int_{S_\F(X_0, R)}  \langle A_\F(X) (\nabla \F(X) - \nabla \F(X_0)), \nabla \F(X) - \nabla \F(X_0) \rangle \dX\\\nonumber
& \leq (n +2) K_d R \,\muF(S_\F(X_0, R)).
\end{align}
Notice that $\Phi \notin C^3(\re^{n+1})$, so we cannot directly apply \eqref{Dxx0:phi} with $\phi = \F$. Instead, we will use the tensorial nature of $\Phi$. Given a section $S_\Phi:=S_\F(X_0, R)$, with $X_0=(x_0, z_0) \in \re^{n+1}$, by means of  \eqref{exp:nablaF:H:2} and \eqref{exp:delta:F} we can write
\begin{align*}
&\int_{S_\F}  \langle A_\F(X) (\nabla \F(X) - \nabla \F(X_0)), \nabla \F(X) - \nabla \F(X_0) \rangle \dX\\
& =\int_{S_\F} |\nabla^\F \delta_\F(X_0, X)|^2 \dmuF(X) \\
&= \int_{S_\F} \left( \langle  D^2 \f(x)^{-1} \nabla \delta_\F(X_0, X), \nabla \delta_\F(X_0, X) \rangle +   |z|^{2-1/s} \frac{\partial{\delta_\F}}{\partial z}(X_0, X)^2 \right) \dmuF(X)\\
&= \int_{S_\F} \left( |\nabla^\f \delta_\f(x_0, x)|^2 +   |z|^{2-1/s} (h_s'(z) - h_s'(z_0))^2 \right) \dmuF(X).
\end{align*}
Now, from the first inclusion in \eqref{S:vs:SxS} and  \eqref{Dxx0:phi} (used with $\phi = \f$, since $\f \in C^3(\rn)$), 
\begin{align*}
& \int_{S_\F}  |\nabla^\f \delta_\f(x_0, x)|^2  \dmuF(X) \leq \int_{S_\f(x_0, R)}  |\nabla^\f \delta_\f(x_0, x)|^2 \dmuf(x)
\times \int_{S_{h_s}(z_0, R)} h_s''(z) \dz\\
& \leq n R \muf(S_\f(x_0, R)) \mu_{h_s}(S_{h_s}(z_0, R)) = n R\muF(S_\f(x_0, R) \times S_{h_s}(z_0, R) )\\
& \leq n R \muF(S_\F(X_0, 2R)) \leq n K_d R \muF(S_\F(X_0, R)),
\end{align*}
where for the last two inequalities above we used the second inclusion in \eqref{S:vs:SxS} and the doubling property \eqref{def:Kd}. On the other hand, by \eqref{def:muF},
\begin{align*}
 \int_{S_\F}  |z|^{2-1/s} (h_s'(z) - h_s'(z_0))^2 \dmuF(X) =  \int_{S_\F} (h_s'(z) - h_s'(z_0))^2 \dmuf(x) \dz.
\end{align*}
Let us write the one-dimensional section $S_{h_s}(z_0, R)$ as $S_{h_s}(z_0, R) = (z_\ell, z_r)$, where $z_\ell, z_r \in \re$ satisfy 
\begin{equation}\label{cond:zrzl}
h_s(z_\ell) - h_s(z_0) - h_s'(z_0)(z_\ell - z_0) = h_s(z_r) - h_s(z_0) - h_s'(z_0)(z_r - z_0) = R.
\end{equation}
From  the first inclusion in \eqref{S:vs:SxS}  we have
\begin{align*}
 \int_{S_\F} (h_s'(z) - h_s'(z_0))^2 \dmuf(x) \dz \leq \muf(S_\f(x_0, R)) \int_{z_\ell}^{z_r}  (h_s'(z) - h_s'(z_0))^2 \dz.
\end{align*}
As $h_s'$ is increasing, 
\begin{align*}
 \int_{z_\ell}^{z_r}  (h_s'(z) - h_s'(z_0))^2 \dz &\leq (h_s'(z_r) - h_s'(z_\ell))  \int_{z_\ell}^{z_r}  (h_s'(z) - h_s'(z_0))^2 \dz\\
&  = \mu_{h_s}(z_\ell, z_r)  \int_{z_\ell}^{z_r}  |h_s'(z) - h_s'(z_0)| \dz.
\end{align*}
At this point we split the last integral above as
\begin{align*}
& \int_{z_\ell}^{z_r}  |h_s'(z) - h_s'(z_0)| \dz  = \int_{z_\ell}^{z_0} ( h_s'(z_0) -  h_s'(z)) \dz + \int_{z_0}^{z_r} (h_s'(z) - h_s'(z_0)) \dz\\
& = \left( h_s'(z_0)(z_0 - z_\ell) - h_s(z_0) + h_s(z_\ell)  \right)+  \left( h_s(z_r) - h_s(z_0) - h_s'(z_0)(z_r - z_0)  \right)= 2R,
\end{align*}
where the last equality is due to \eqref{cond:zrzl}. Therefore,
\begin{align*}
& \int_{S_\F}  |z|^{2-1/s} (h_s'(z) - h_s'(z_0))^2 \dmuF(X) \leq \muf(S_\f(x_0, R)) \int_{z_\ell}^{z_r}  (h_s'(z) - h_s'(z_0))^2 \dz\\
& \leq \muf(S_\f(x_0, R))  \mu_{h_s}(z_\ell, z_r)  \int_{z_\ell}^{z_r}  |h_s'(z) - h_s'(z_0)| \dz =  \muf(S_\f(x_0, R)) \mu_{h_s}(S_{h_s}(z_0, R)) 2 R\\
& \leq 2 R \muF(S_\F(X_0, 2R)) \leq 2 K_d R  \muF(S_\F(X_0, R)),
\end{align*}
and \eqref{L2:Energy:delta:Phi} follows. \qed 

\section{The function $\F$ and a weak Poincar\'e inequality}\label{sec:PI:F} 

Our goal in this section is to prove a version of the Poincar\'e inequality \eqref{Poincare:phi} with Lebesgue measure being replaced by the Monge--Amp\`ere measure $\muF$. We will reason along the lines of \cite[Section 4]{MaMRL13}, where the change in the opposite direction (i.e. from Monge--Amp\`ere to Lebesgue measure) was made by means of convex conjugation. In the case of $\F$, however, an approximation argument will be used to circumvent the fact that $\F \notin C^2(\re^{n+1})$. 

\begin{thm}\label{thm:Poincare:Phi}
Let $\F$ be as in \eqref{def:Phi:s}. Then there exist geometric constants $K_2 > 1$ and $K_{P} > 0$, such that for every section $S_\F:=S_\F(X_0, R)$ and $G \in C(S_\F(X_0, K_2 R))$ with $\nabla^\F G \in L^2(S_\F(X_0, K_2 R), \dmuF)$ we have
\begin{equation}\label{Poincare:Phi}
\fint_{S_\F} |G(X) - G_{S_\F}| \dmuF(X)  \leq K_{P} \, R^\frac{1}{2} \left(\fint_{S_\F(X_0, K_2 R)} |\nabla^\F G(X)|^2 \dmuF(X) \right)^\frac{1}{2},
\end{equation}
where
$$G_{S_\F}:= \fint_{S_\F} G(X) \dmuF(X).$$
\end{thm}

\begin{proof}  Let $\eta \in C^1(\re)$ be nonnegative and compactly supported in $[-1,1]$ with $\int_\re \eta =1$. For $\eps > 0$ set $\eta_\eps(z):=\tfrac{1}{\eps}\eta(\tfrac{z}{\eps})$ and define
 \begin{equation*}\label{def:hs:eps}
 h_{s, \eps}(z):= h_s * \eta_\eps(z) \quad \forall z \in \re
 \end{equation*}
and
\begin{equation*}\label{def:Phi:eps}
\F_{\eps}(x,z):= \f(x) + h_{s, \eps}(z) \quad \forall (x,z) \in \rn \times \re.
\end{equation*}
Then, for every $\eps > 0$,  we have that $\F_{\eps} \in C^2(\re^{n+1})$ with $D^2 \F_{\eps} > 0$. In addition,  $h_{s, \eps}''$ converges to $h_s''$ in $L^1_{\mathrm{loc}}(\re)$ and, consequently, $\mu_{\F_{\eps}}$ converges to $\muF$ in $L^1_{\mathrm{loc}}(\re^{n+1})$ as $\eps$ tends to $0$. Also, $\nabla \F_{\eps}$ converges to $\nabla \F$ uniformly on compact sets of $\re^{n+1}$ .  

The matrix of cofactors of $D^2 \F_{\eps}(x,z)$ is given by 
\begin{equation*}\label{def:AFs:eps}
A_{\F_{\eps}}(x,z):=
\begin{pmatrix}
A_\v(x)h_{s, \eps}''(z)& 0 \\
0 & \muf(x)
\end{pmatrix}
\in \R^{n+1}\times\R^{n+1}.
\end{equation*}
Now, a simple computation shows that, for each $\eps > 0$,  the measure $h_{s, \eps}''$ is a doubling measure on the real line with doubling constant smaller than or equal to the doubling constant for $h_s''$, which, in turn, depends only on $s$. Indeed, let $C_s \geq 1$ denote the doubling constant for $h_s''$ as measure on the real line. Given $c \in \re$ and $r > 0$, let $I_c:=[c-r, c+r]$ and $2I_c:= [c- 2r, c+ 2r]$. Then, 
\begin{align*}
\int_{2I_c} h_{s, \eps}''(z) \dz = \int_{2I_c}  \int_{\re} \eta_\eps(y) h_s''(z-y) \dy \dz =   \int_{\re} \int_{2I_c}  \eta_\eps(y) h_s''(z-y)  \dz \dy.
\end{align*}
For each $y \in \re$, by changing variables $w:= y-z$ we get $z \in 2 I_c$ if and only if $w \in 2 I_{c-y}$. Hence, using that $h_s''$ is even and doubling with constant $C_s$,
\begin{align*}
&\int_{\re} \int_{2I_c}  \eta_\eps(y) h_s''(z-y)  \dz \dy = \int_{\re} \eta_\eps(y) \int_{2 I_{c -y}} h_s''(w) \, dw\\
& \leq C_s  \int_{\re} \eta_\eps(y) \int_{I_{c -y}} h_s''(w) \, dw = C_s \int_{I_c}  \int_{\re} \eta(y) h_s''(y-z) \dy \dz = C_s\int_{I_c} h_{s, \eps}''(z) \dz.
\end{align*}
Thus,
$$
\int_{2I_c} h_{s, \eps}''(z) \dz \leq C_s\int_{I_c} h_{s, \eps}''(z) \dz,
$$
uniformly in $\eps > 0$.  By \cite[Lemma 6]{FM09}, it follows that the $\dcc$ constants for  $\F_{\eps}$ are controlled by the one for $\F$ uniformly in $\eps > 0$.

Next, let $\Psi_{\eps}$ denote the convex conjugate of $\F_{\eps}$. By \eqref{psi:C2}   we have  $\Psi_{\eps} \in C^2(\re^{n+1})$ with $D^2 \Psi_{\eps} > 0$ for every $\eps > 0$. 

Fix $Y_0 \in \re^{n+1}$, $R > 0$, and set $X_0:=\nabla \F_{\eps}(Y_0)$. Let $\kappa_1, K_1$ be the geometric constants from \eqref{sec:phi:phi} applied to $\F_{\eps}$ and define  $S_{\Psi_{\eps}}:= S_{\Psi_{\eps}}(Y_0, R/\kappa_1)$ and $K_2:=K_1/\kappa_1$. Hence, the inclusions \eqref{sec:phi:phi} yield
\begin{equation}\label{SFeps:SPhi}
S_{\F_{\eps}}(X_0, R) \subset S^{\F_{\eps}} := \nabla \Psi_{\eps} (S_{\Psi_{\eps}}) \subset S_{\F_{\eps}}(X_0, K_2 R).
\end{equation}
Assume first that $G \in C^1(S_\F(X_0, K_2 R))$ and  define $H \in C^1(S_{\Psi_{\eps}})$ as 
$$
H(Y):= G(\nabla \Psi_\eps(Y)) \quad \forall Y \in S_{\Psi_{\eps}}
$$
which, by putting $Y:= \nabla \F_\eps (X)$, yields
\begin{align*}
\nabla^{\Psi_\eps}H(Y)& = D^2 \Psi_\eps(Y)^{-1/2} \nabla H(Y) = D^2 \Psi_\eps(Y)^{-1/2} D^2 \Psi_\eps(Y) \nabla G(X)\\
&= D^2 \Psi_\eps(Y)^{1/2}  \nabla G(X)= D^2 \F_\eps(X)^{-1/2} \nabla G(X) = \nabla^{\F_\eps} G(X).
\end{align*}
Now, by the Poincar\'e inequality \eqref{Poincare:phi} applied to $\Psi_{\eps}$ on the section $S_{\Psi_{\eps}}$ and $H \in C^1(S_{\Psi_{\eps}})$, we get
\begin{equation}\label{Poincare:Psi:eps}
\fint_{S_{\Psi_{\eps}}} |H(Y) - H_{S_{\Psi_{\eps}}}| \dY  \leq C_{P} \, R^\frac{1}{2} \left(\fint_{S_{\Psi_{\eps}}} |\nabla^{\Psi_{\eps}} H(Y)|^2 \dY \right)^\frac{1}{2},
\end{equation}
with $C_P > 0$ a geometric constant,  and  by changing variables $Y:= \nabla \F_{\eps}(X)$ in \eqref{Poincare:Psi:eps} and using that the identity \eqref{psi:phi:inv} gives $dY= \mu_{\F_{\eps}}(X) \dX$ and
$$ 
|\nabla^{\Psi_{\eps}} H(Y)|^2 dY = |\nabla^{\F_{\eps}} G(X)|^2   \mu_{\F_{\eps}}(X) \dX = \langle A_{\F_\eps}(X) \nabla G(X), \nabla G(X) \rangle \dX. 
$$
we obtain
\begin{equation}\label{Poincare:Phi:eps}
\fint_{S^{\F_{\eps}}} |G(X) - G^{S^{\F_{\eps}}}| \,  \mu_{\F_{\eps}}(X) \dX
\leq \left(\frac{ C_{P}^2 \, R}{\mu_{\F_{\eps}}(S^{\F_{\eps}})} \int_{S^{\F_{\eps}}}|\nabla^{\F_{\eps}} G(X)|^2 \, \mu_{\F_{\eps}}(X) \dX \right)^\frac{1}{2},
\end{equation}
where we have put
$$
G^{S^{\F_{\eps}}}:= \frac{1}{\mu_{\F_{\eps}}(S^{\F_{\eps}})  }\int_{S^{\F_{\eps}}} G(X) \, \mu_{\F_{\eps}}(X) \dX
$$
and used that $|S_{\Psi_{\eps}}| = |\nabla \F_{\eps}(S^{\F_{\eps}})| = \mu_{\F_{\eps}}(S^{\F_{\eps}})$, due to the definition of $S^{\F_{\eps}}:= \nabla \Psi (S_{\Psi_{\eps}})$ in \eqref{SFeps:SPhi}. Next,  the inclusions \eqref{SFeps:SPhi} and the doubling property \eqref{doubling:r:R} give
\begin{align*}
\frac{1}{K_d K_2^\nu \mu_{\F_{\eps}}(S_{\F_{\eps}}(X_0, R))} & \leq \frac{1}{\mu_{\F_{\eps}}( S_{\F_{\eps}}(X_0, K_2 R))} \leq \frac{1}{\mu_{\F_{\eps}}(S^{\F_{\eps}})}\\
&\leq \frac{1}{\mu_{\F_{\eps}} (S_{\F_{\eps}}(X_0, R))} \leq \frac{K_d K_2^\nu}{\mu_{\F_{\eps}}( S_{\F_{\eps}}(X_0, K_2 R))}.
\end{align*}
Consequently, the inclusions \eqref{SFeps:SPhi}, the inequality \eqref{Poincare:Phi:eps}, and the fact that $A_{\F_\eps}$ converges to $A_\F$ in $L^1_{loc}(\re^{n+1})$, imply \eqref{Poincare:Phi} with $K_P:=2 C_P (K_d K_2)^{3/2}$ in the case $G \in C^1(S_\F(X_0, K_2 R))$. Then,  \eqref{Poincare:Phi}  in the case of $\nabla^\F G \in L^2(S_\F(X_0, K_2 R), \dmuF)$ follows by approximation in $L^2(S_\F(X_0, K_2 R), \dmuF)$. Just notice that smooth functions are dense in $L^2(\Omega, w(X)dX)$ for every open, bounded, convex subset $\Omega \subset \re^{n+1}$ and every $w \in A_\infty(\Omega)$ and that on the compact set $\overline{\Omega}$ we have $\muF(x,z) \sim |z|^{1/s -2} \in A_{\infty}(\Omega)$. \end{proof}

Next we record a simple consequence of the weak Poincar\'e's inequality known as ``Fabes lemma''. 

\begin{cor}\label{cor:Fabes} Let $\F$ be as in \eqref{def:Phi:s} and let $K_P > 0$ and $K_2 >1$ be the geometric constants from Theorem \ref{thm:Poincare:Phi}. Then, for every $\eps \in (0,1)$, every section $S:= S_\F(X_0, R)$ and every  $G \in C(S_\F(X_0, K_2 R))$ with $\nabla^\F G \in L^2(S_\F(X_0, K_2 R), \dmuF)$ and
\begin{equation*}
\muF(\{X \in S: G(X) = 0 \}) \geq \eps \muF(S),
\end{equation*}
we have that
\begin{equation}\label{Fabes:ineq}
\fint_S |G(X)| \dmuF(X)  \leq  \left(1+ \frac{1}{\eps} \right) K_P  \left(\fint_{S_\F(X_0, K_2 R)} |\nabla^\F G(X)|^2 \dmuF(X) \right)^\frac{1}{2}.
\end{equation}
 \end{cor}

\begin{proof} Setting $E:= \{X \in S: G(X) = 0 \}$ and $G_E:= \fint_E G \dmuF =0$, for $X \in S$ we have
\begin{align*}
|G(X)| &\leq |G(X)- G_S| + |G_S - G_E| \leq |G(X)- G_S| + \left(\frac{\muF(S)}{\muF(E)}\right)  \fint_S |G - G_S| \dmuF\\
& \leq  |G(X)- G_S| +  \frac{1}{\eps} \fint_S |G - G_S| \dmuF
\end{align*}
so that
\begin{align*}
\fint_S |G(X)| \dmuF(x) \leq  \left(1+ \frac{1}{\eps} \right) \fint_S |G(X) - G_S| \dmuF(X)
\end{align*}
and then \eqref{Fabes:ineq} follows from the weak-Poincar\'e's inequality \eqref{Poincare:Phi}.
\end{proof}

\section{The class $\calS(Q)$}\label{sec:S(Q)}

From here on we fix an arbitrary $p_0 \in \rn$ and $R_0 > 0$ and define the cylinder
\begin{equation*}\label{def:Q}
Q:= S_\f(p_0, R_0) \times \re \subset \re^{n+1}.
\end{equation*}
Also, let us put
\begin{align*}
\quad Z^+&:=\{ (x, z) \in \re^{n+1}: x \in \rn, z > 0 \},\\
Z^-&:=  \{ (x, z) \in \re^{n+1}: x \in \rn, z < 0 \},\\
Z_0:&= \{ (x, 0) \in \re^{n+1}: x \in \rn \},
\end{align*}
and define
$$
Q^+:= Q \cap Z^+ \quad \text{and} \quad Q^-:= Q \cap Z^-.
$$

Henceforth, all of our sub- and super-solutions will belong to a class $\calS(Q)$ modeled after some key properties of $\widetilde{V}(x,z):=V(x, |z|)$ where $V$, as in \eqref{def:V}, is the solution to the extension problem \eqref{EP:nonDiv:S} in $S_0:=S_\f(p_0, R_0)$ for a nonnegative $v \in \Dom_{S_0}(L_\f)$. 

\subsection{The definition of $\calS(Q)$}

We write $F \in \calS(Q)$ if the following conditions hold.
\begin{enumerate}[(i)]
\item\label{z:symmetry} $F(x, z) = F(x, -z)$ for every $(x,z) \in Q$;
\item\label{C:C2} $F \in C(Q) \cap C^2(Q \setminus Z_0)$, in particular, $\nabla F$, and then $\nabla^\F F$, exist a.e. in $Q$;
\item\label{finite:L2:E} for every section $S_\F(X_0, R)$ with $S_\F(X_0, 2R) \subset \subset Q$ we have 
\begin{equation*}
\nabla^\F F \in L^2(S_\F(X_0, 2R), \dmuF);
\end{equation*}
\item\label{normal:derivative} for every $x \in S_\f(p_0, R_0)$ the limit
\begin{equation}\label{def:Fz0+}
F_z(x,0^+):= \lim\limits_{z \to 0^+} F_z(x,z)
\end{equation}
exists and the function $x \mapsto F_z(x,0^+)$ is continuous in $S_\f(p_0, R_0)$. Notice that if the limit in \eqref{def:Fz0+} exists, then we have 
\begin{equation*}\label{def2:Fz}
F_z(x,0^+) = \lim\limits_{z \to 0^+} \frac{F(x, z) - F(x,0)}{z}.
\end{equation*}
\end{enumerate}

\subsection{The definition of $F_{z,0^+}$}

Given $F \in \calS(Q)$ and an open subset $\Omega \subset \subset Q$, define
\begin{equation}\label{def:F0+}
F_{z,0^+}(\Omega):= \left\{
      \begin{array}{lcl}
       \sup \left\{|F_z(x,0^+)|: (x,0) \in \Omega \cap Z_0 \right\} &  \text{if } &  \Omega \cap Z_0\neq \emptyset, \\
          0 &  \text{if } &   \Omega \cap Z_0=  \emptyset.
      \end{array}
    \right.
\end{equation}
The continuity of the function $x \mapsto F_z(x,0^+)$ in $S_\f(p_0, R_0)$ ensures that $F_{z,0^+}(\Omega) < \infty$ for every open subset $\Omega \subset \subset Q$. 

\subsection{The fact that $\widetilde{V} \in \calS(Q)$}

Let  $V$ be the solution to the extension problem \eqref{EP:nonDiv:S} in the section $S_\f(p_0, R_0)$ and set $\widetilde{V}(x,z):=V(x, |z|)$.  Then $\widetilde{V}$ satisfies that $\widetilde{V} \in C(Q) \cap C^2(Q \setminus Z_0)$, $\widetilde{V}$ is even with respect to $z$, and for every $x \in S_\f(p_0, R_0)$,
\begin{equation*}
\widetilde{V}_z(x,0^+):= \lim\limits_{z \to 0^+} \widetilde{V}_z(x,z) = \lim\limits_{z \to 0^+} V_z(x,z) = - d_s L_\f^sv(x),
\end{equation*}
with $L_\f^sv \in C_0(S_\f(p_0, R_0))$ for every $v \in \Dom_S(L_\f)$. Also, given $X_0 = (x_0, z_0) \in \re^{n+1}$ and $R > 0$ such that  $S_\F(X_0, 2R) \subset \subset Q$, the inclusions \eqref{S:vs:SxS} yield
\begin{equation*}
S_\Phi(X_0, R) \subset S_\f(x_0, R) \times S_{h_s}(z_0, R) \subset S_\Phi(X_0, 2R)  \subset \subset Q.
\end{equation*}
Now,  by recalling  \eqref{exp:nablaF:H:2}, to get an expression for $ |\nabla^\F \widetilde{V}(X)|^2$, and by the energy estimate \eqref{E:E:V} from Remark \ref{rmk:finite:E}, we obtain
\begin{align*}
& \int_{S_\Phi(X_0, R) } |\nabla^\F \widetilde{V}(X)|^2 \dmuF(X) \leq  \int_{Q} |\nabla^\F \widetilde{V}(X)|^2 \dmuF(X) \\
& = \int_{S_\f(p_0, R) }  \int_{-\infty}^{\infty} \left( |\nabla^\f \widetilde{V}(x,z)|^2 +   |z|^{2-1/s} \widetilde{V}_z(x,z)^2\right)  |z|^{1/s-2} \dz \dmuf(x)\\
& = 2 \int_{S_\f(p_0, R) }  \int_{0}^{\infty} |\nabla^\f V(x,z)|^2  z^{1/s-2} \dz \dmuf(x) + 2 \int_{S_\f(p_0, R) }  \int_{0}^{\infty} V_z(x,z)^2 \dz \dmuf(x) < \infty.
\end{align*}
Hence, $\widetilde{V} \in \calS(Q)$. 

Notice that $\calS(Q)$ is a vector space that contains the constant functions. Also, no vanishing condition on $\partial S_\f(p_0, R_0)$ is prescribed in the definition of $F \in \calS(Q)$, although this is the case for $\widetilde{V}.$

\section{The critical-density estimate}\label{sec:CD} 

For $s \in (0,1)$ introduce
\begin{equation}\label{def:Ms:betas}
M_s:=  \frac{2^{2+s} (1-s)^s}{s^{2s}} + 8 \quad \text{and} \quad \beta_s:= 4 K_s(1 + 8 K_s),
\end{equation}
where $K_s \geq 1$ is the quasi-triangle constant in \eqref{def:Ks}, which depends only on $s$. The main result in this section is the following critical-density estimate. 

\begin{thm}\label{thm:CD}  Let $\F$ be as in \eqref{def:Phi:s}.  There exist geometric constants $\theta_0, \eps_0 \in (0,1)$ such that for every section $S_R:=S_\F(X_0, R)$ with $S_{\beta_s R}:= S_\F(X_0, \beta_s R) \subset Q$ and every $W \in \calS(Q)$ with  $L_\F W =  - \tr{( (D^2 \F)^{-1} D^2 W}) \geq 0$ pointwise in $Q^+$, the inequalities
\begin{equation*}\label{W0Rs<theta0}
W_{z,0^+}(S_{\beta_s R}) R^s \leq \theta_0
\end{equation*}
and
\begin{equation*}\label{hyp:W<1:CD}
\inf\limits_{S_R} W \leq 1
\end{equation*}
imply
\begin{equation*}\label{muF(W<Ms:CD}
\muF(\{X \in S_{\beta_s R}: W(X) < M_s \}) \geq \eps_0 \muF(S_{\beta_s R}).  
\end{equation*}
\end{thm}

As mentioned in the introduction, due to the degeneracy of $D^2\F$ on the hyperplane $Z_0$, one cannot directly resort to the proof of Theorem 1 in \cite{Caffarelli-Gutierrez} or the one for Theorem 2 in \cite{MaldonadoJDE14}. In order to address the degeneracy or singularity of $D^2\F$ on the hyperplane $Z_0$, the proof of Theorem \ref{thm:CD} will be broken down into three cases according to the position of the sections of $\F$ with respect to $Z_0$. These three cases are illustrated in Figure \ref{fig:CDsections}.

In the case of sections intersecting $Z_0$, the $L^\infty$-norm of normal derivatives, in the sense of \eqref{def:Fz0+}, of super-solutions will play a key role in addressing the degeneracy or singularity of $L_\F$. 

\begin{center}
\begin{figure}[htp]
\includegraphics[width=5.5in]{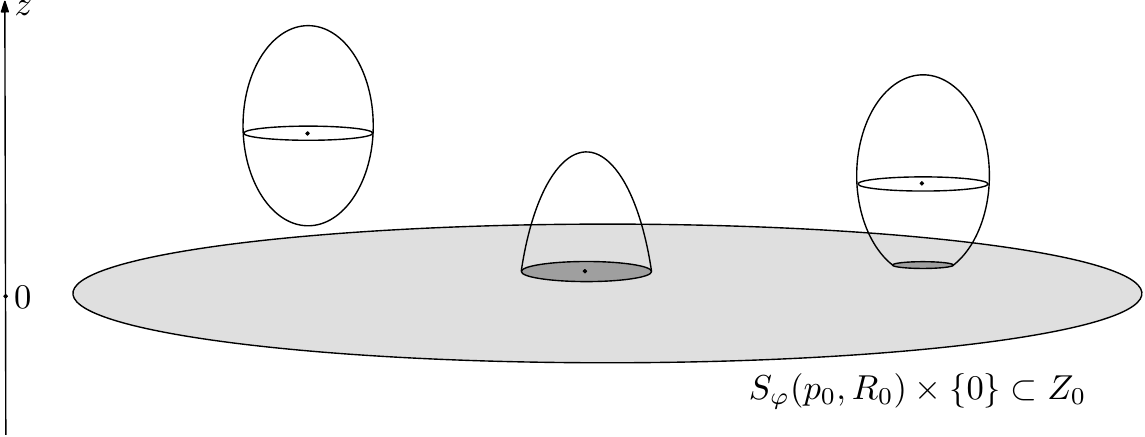}
\caption{The three cases for a section of $\F$ (within $Q$) to be considered in the proof of Theorem \ref{thm:CD}. When the section intersects $Z_0$, the control on the size of the normal derivative of a solution, as defined by \eqref{def:F0+}, will counteract the degeneracy or singularity of $L_\Phi$ at $Z_0$.}\label{fig:CDsections}
\end{figure}
\end{center}

Let us start by stating a version of the Aleksandrov--Bakelman--Pucci maximum principle on which the cases are built. 

\subsection{The ABP maximum principle}

Recall that given a domain $\Omega \subset \rN$ and $u: \Omega \to \re$, the normal mapping of $u$ at $x \in \Omega$, denoted by $\partial u(x)$, is defined as the set
\begin{equation*}
\partial u(x) := \{p \in \rN: u(y) \geq u(x) + \langle p, y-x \rangle \quad \forall y \in \Omega \}
\end{equation*}
and, given $E \subset \Omega$, 
$$
\partial u(E):= \bigcup_{x \in E} \partial u(x). 
$$
Also, let $\Gamma_u$ denote the convex envelope of $u$ in $\Omega$ and let $\mathcal{C}_{u} :=\{x \in \Omega: u(x) = \Gamma_{u}(x) \}$ denote the contact set of $u$ and $\Gamma_{u}$ in $\Omega$. In particular (see \cite[pp.13-16]{guti}), 
\begin{equation}\label{Gamma:C}
\partial \Gamma_{-u}(\calC_{-u}) = - \partial (-u) (\calC_{-u}).
\end{equation}

The next lemma allows to replace $\diam(\Omega)$ in \cite[Theorem 1.4.5]{guti} with $|\Omega|^{1/N}$ when $\Omega$ is bounded and convex (its proof is a combination of the one for Theorem 1.4.5 from \cite[p.15]{guti} and the normalization technique in \cite[Section 1]{Caffarelli-Gutierrez}). 

\begin{lem}\label{lem:ABP1} Let $\Omega \subset \rN$ be open, convex, bounded and let $U \in C(\overline{\Omega})$ satisfy $U \leq 0$ on $\partial \Omega$. Let $\Gamma_{-U}$ and  $\mathcal{C}_{-U}$ denote the convex envelope of $-U$ in $\Omega$ and the contact set of $-U$ with $\Gamma_{-U}$ in $\Omega$, respectively. Then, there exists a dimensional constant $C_N > 0$ such that 
\begin{equation*}
\max\limits_\Omega U \leq C_N |\Omega|^\frac{1}{N} |\partial (\Gamma_{-U})(\mathcal{C}_{-U})|^\frac{1}{N}.
\end{equation*}
\end{lem}

We will be using the following consequence of Lemma \ref{lem:ABP1}. 

\begin{cor}\label{cor:ABP2}  Let $\Omega \subset \rN$ be open, convex, and bounded. Suppose that $H \in C(\overline{\Omega})$ satisfies the following conditions
\begin{enumerate}[(i)]
\item $H \geq 0$ on $\partial \Omega$,
\item there is an open set $\Omega' \subset \Omega$ such that  $H \in C^2(\Omega')$ and  $\calC_{-H^-} \subset \Omega'$, where $H^-:=\max\{0, -H\}$. 
\end{enumerate}
Then,
\begin{equation*}
\max\limits_{\Omega} H^- \leq C_N \, |\Omega|^\frac{1}{N} \left( \int_{\calC_{-H^-}} |\det D^2 H(X)| \,dX\right)^\frac{1}{N}. 
\end{equation*}
\end{cor}

\begin{proof} Lemma \ref{lem:ABP1} applied to $H^-$ yields
\begin{equation}\label{ABP:Gamma:H}
\max\limits_{\Omega} H^- \leq C_N |\Omega|^\frac{1}{N} |\partial (\Gamma_{-H^-})(\calC_{-H^-})|^\frac{1}{N}.
\end{equation}
By \eqref{Gamma:C} with $u = H^-$, 
\begin{equation}\label{GammaH=HH}
\partial (\Gamma_{-H^-}) (\calC_{-H^-}) = - \partial (-H^-) (\calC_{-H^-}).
\end{equation}
Now, for $X \in\calC_{-H^-}$ the fact that $H^- = 0$ on $\partial \Omega$ yields $-H^-(X) = H(X) \leq 0$ (with $H(X)<0$ unless $H \equiv 0$). Let us see that this implies the inclusion
\begin{equation}\label{incl:HH}
\partial (-H^-) (\calC_{-H^-}) \subset  \partial H (\calC_{-H^-}).
\end{equation} 
Indeed, given $X \in \calC_{-H^-}$ and $P \in \partial (-H^-)(X)$ the definition of normal mapping gives
\begin{equation}\label{P:in:HX}
- H^-(Y) \geq -H^-(X) + \langle P, Y- X \rangle \quad \forall Y \in \Omega,
\end{equation}
and since  $-H^-(X) = H(X)$ and $- H^-(Y) \leq H(Y)$ for every $Y \in \Omega$, the inequality \eqref{P:in:HX} implies that $P \in \partial H(X)$, thus proving \eqref{incl:HH}. By combining \eqref{GammaH=HH} and \eqref{incl:HH}, we have
$$
\partial (\Gamma_{-H^-}) (\calC_{-H^-}) \subset - \partial H (\calC_{-H^-}),
$$ 
so that, from \eqref{ABP:Gamma:H}, 
\begin{equation*}
\max\limits_{\Omega} H^- \leq C_N |\Omega|^\frac{1}{N} |\partial H (\calC_{-H^-})|^\frac{1}{N}.
\end{equation*}
Finally, by the assumptions $\calC_{-H^-} \subset \Omega'$ and $H \in C^2(\Omega')$, the inequality
\begin{equation*}
|\partial H (\calC_{-H^-})| \leq \int_{\calC_{-H^-}} |\det D^2 H(X)|\, dX
\end{equation*}
follows from the usual formula of change of variables and the proof is complete.
\end{proof}

\subsection{The case of sections of $\F$ away from $Z_0$.}

\begin{thm}\label{thm:CD:Case1} 
 Let $\F$ be as in \eqref{def:Phi:s}. Let $X_0 \in \rn \times \re$ and $R > 0$ such that $S_\F( X_0, 2R) \subset \subset Q$. Set $S_R:= S_\F( X_0, R)$, $S_{2R} := S_\F( X_0, 2R)$ and suppose that $\overline{S_{2R}} \cap Z_0 = \emptyset$. Then, there exists a geometric constant $\eps_1 \in (0,1)$ such that for every $W \in \calS(Q)$ with $L_\F W = - \tr{( (D^2 \F)^{-1} D^2 W})\geq 0$ and $W \geq 0$ in $\overline{S_{2R}}$, the inequality
\begin{equation}\label{cond:W<1}
\inf\limits_{S_R} W \leq 1
\end{equation}
implies
\begin{equation}\label{muF(W<4)}
\muF(\{X \in S_{2R}: W(X) < 4  \}) \geq \eps_1 \muF(S_{2R}).  
\end{equation}
\end{thm}

\begin{proof} Introduce the auxiliary function
\begin{equation*}\label{def:v}
H(X):= W(X) + 4 \left(\frac{\delta_\F(X_0, X)}{2R} - 1 \right)\quad \forall X \in \re^{n+1},
\end{equation*}
where the function
$$
X \mapsto \delta_\F(X_0, X):= \F(X) - \F(X_0) - \langle \nabla \F(X_0), X - X_0 \rangle,
$$
is convex with $D^2  \delta_\F(X_0, \cdot) = D^2 \F$. By \eqref{cond:W<1}, there is $X_1 \in S_\F(X_0, R)$ such that $W(X_1) < 1$, and then
$$
H(X_1) = W(X_1) + 4  \left(\frac{\delta_\F(X_0, X_1)}{2R} - 1 \right) < 1 + 4 \left(\frac{1}{2} - 1\right) = -1.
$$
Now, the ABP maximum principle in Corollary \ref{cor:ABP2} applied to the function $H^-$ on the convex set $\Omega =\Omega'=S_{2R}$  (notice that $H = W \geq 0$ on $\partial S_{2R}$) yields
\begin{equation*}
1 < H^-(X_1)^{n+1} \leq C_n |S_{2R}| \int_{\calC_{-H^-}} |\det D^2 H(X)| \dX.
\end{equation*}
In particular, on the contact set $\calC_{-H^-}$ we have that $H$ is negative, and then
\begin{align*}
\calC_{-H^-} \subset \{X \in S_{2R}: H(X) < 0\}& =  \{ X \in S_{2R}: W(X) < 4 (1 -\delta_\F(X_0,X)/2R)\}\\
& \subset \{ X \in S_{2R}: W(X) < 4\}.
\end{align*}
Also, $D^2 H \geq 0$ on $\calC_{-H^-}$ so that, recalling that $A_\F:= (\det D^2 \F) D^2 \F^{-1}$ is positive in $S_{2R}$,  on $\calC_{-H^-}$ we have
\begin{align*}
0 \leq \det D^2 H &= \det (A_\F D^2 H) (\det A_\F)^{-1} \leq \left[\frac{\tr(A_\F D^2 H)}{n+1} \right]^{n+1}  (\det D^2 \F)^{-n}\\
& = \left[\frac{ \tr(A_\F D^2 W) + \frac{2}{R} \tr(A_\F D^2\F) }{n+1} \right]^{n+1}  (\det D^2 \F)^{-n}\\
& \leq \left(\frac{ 2 \det D^2 \F }{R} \right)^{n+1}  (\det D^2 \F)^{-n} = \left(\frac{ 2 }{R} \right)^{n+1}  \det D^2 \F.
\end{align*}
Therefore,
\begin{align*}
1& \leq C_n |S_{2R}| \left(\frac{ 2 }{R} \right)^{n+1}  \int_{\calC_{-H^-}}  \det D^2 \F(X) \dX \\
& \leq C_n |S_{2R}| \left(\frac{ 2 }{R} \right)^{n+1}  \muF(\{X \in S_{2R}: W(X) < 4\}).
\end{align*}
 Then, by recalling the definition of $K_3$ in \eqref{def:K3}, 
\begin{align*}
1 &\leq  \frac{C_n 4^{n+1} K_3}{\muF(S_{2R})} \muF(\{X \in S_{2R}: W(X) < 4\})
\end{align*}
and \eqref{muF(W<4)} follows with $\eps_1:= (C_n 4^{n+1} K_3)^{-1}$.
\end{proof}

\subsection{The case of sections of $\F$ centered at $Z_0$.}

For each $s \in (0,1)$, set
\begin{equation*}\label{def:qs}
q_s :=  \frac{2^s (1-s)^s}{s^{2s}}.
\end{equation*}

\begin{thm}\label{thm:CD:center:Z0}
Let $\F$ be as in \eqref{def:Phi:s}. Let $x_0 \in \rn$ and $R > 0$ such that $S_\F( (x_0, 0), 2R) \subset \subset Q$. Put $S_R:= S_\F( (x_0, 0), R)$ and $S_{2R} := S_\F( (x_0, 0), 2R)$. Suppose that $W \in \calS(Q)$ satisfies $L_\F W =  - \tr{( (D^2 \F)^{-1} D^2 W})  \geq 0$ and $W \geq 0$ pointwise in $S_{2R} \setminus Z_0$. 

Then, there exist geometric constants $\theta_2, \eps_2 \in (0,1)$ such that the inequalities
\begin{equation}\label{W0Rs<eps1}
W_{z,0^+}(2R) R^s \leq \theta_2
\end{equation}
and
\begin{equation}\label{hyp:W<1}
\inf\limits_{S_R} W \leq 1
\end{equation}
imply
\begin{equation}\label{muF(W<4qs8}
\muF(\{X \in S_{2R}: W(X) < 4 q_s + 8 \}) \geq \eps_2 \muF(S_{2R}).  
\end{equation}
\end{thm}

\begin{proof}
Notice that the expression for $\delta_\F$ in \eqref{exp:delta:F} and the fact that $h_s(0) = h_s'(0) =0$ give 
\begin{align*}
\delta_\F((x_0, 0), X) = \f(x) - \f(x_0) - \langle \nabla \f(x_0), x-x_0 \rangle + h_s(z) \quad \forall X =(x,z) \in \re^{n+1},
\end{align*}
which makes all sections centered at $Z_0$ symmetric with respect with $z$. Let
\begin{equation}\label{def:Qs}
Q_s:= 4 q_s W_{z,0^+}(2R) R^s + 8. 
\end{equation}
For $X =(x,z) \in \overline{S_{2R}}$, introduce the auxiliary function $H \in C(\overline{S_{2R}})$ as
$$
H(X) := W(X) + Q_s  \left(\frac{\delta_\F((x_0, 0), X)}{2R} - 1 \right) - W_{z,0^+}(2R) |z|- \frac{|z|}{q_s R^s} + W_{z,0^+}(2R) q_s R^s  + 1
$$
which makes $H$ symmetric in $z$ as well (that is, $H(x,z) = H(x, -z)$ for every $(x,z) \in S_{2R}$. Now, by the inclusions \eqref{S:vs:SxS}, 
$$
S_{2R} \subset S_\f(x_0, 2R) \times S_{h_s}(0, 2R)
$$
so that  $X =(x,z) \in \overline{S_{2R}}$ implies $z \in \overline{S_{h_s}(0, 2R)}$ and then $h_s(z) \leq 2R$ (because $h_s(0) = h_s'(0) =0$). Consequently, from the definition of $h_s(z)$, $X =(x,z) \in \overline{S_{2R}}$ implies
\begin{equation}\label{z<qsRs}
|z| \leq \frac{2^s (1-s)^s}{s^{2s}} R^s =: q_s R^s.
\end{equation}
In particular,  for $X =(x,z) \in \partial S_{2R}$ (where it holds that $\delta_\F((x_0, 0), X) = 2R$) and using that $W \geq 0$ and
\begin{equation}\label{01:order>0}
 - W_{z,0^+}(2R) |z| + W_{z,0^+}(2R) q_s R^s \geq 0 \quad \text{and} \quad - \frac{|z|}{q_s R^s}   + 1 \geq 0,
\end{equation}
we have
$$H(X)= W(X)  - W_{z,0^+}(2R) |z|- \frac{|z|}{q_s R^s} + W_{z,0^+}(2R) q_s R^s  + 1\geq 0.$$
On the other hand, for every $(x, 0) \in S_{2R} \cap Z_0$ (and using $h_s'(0)=0$ again), 
$$
\frac{\partial H}{\partial z^+}(x, 0) =\frac{\partial W}{\partial z^+}(x, 0)  - W_{z,0^+}(2R) - \frac{1}{q_s R^s} \leq - \frac{1}{q_s R^s} < 0.
$$
As a consequence, the convex envelope of $-H^-$ in $S_{2R}$ cannot touch $-H^-$ on $Z_0$; moreover, the contact set $\calC_{-H^-}$ lies at a positive distance from $Z_0$.  Therefore, $\calC_{-H^-} \subset S_{2R} \setminus Z_0$. By Corollary \ref{cor:ABP2} applied to $H$ with $\Omega'= S_{2R} \setminus Z_0$ and $\Omega = S_{2R}$ (notice that $H \in C^2(S_{2R} \setminus Z_0)$)  we obtain
\begin{equation}\label{ABP:H:S2R}
\Big(\max\limits_{S_{2R}} H^-\Big)^{n+1}\leq C_n \, |S_{2R}| \int_{\calC_{-H^-}} |\det D^2 H(X)| \,dX. 
\end{equation}
For $X$ in the contact set $\calC_{-H^-} \subset S_{2R} \setminus Z_0$ we have
\begin{align*}
|\det D^2 H(X)|& = \muF(X) \det D^2 \F(X)^{-1} |\det D^2 H(X)|\\
& \leq \muF(X) \left(\frac{\tr(D^2 \F(X)^{-1} D^2 H(X))}{n+1}\right)^{n+1}.
\end{align*}
In $S_{2R} \setminus Z_0$ we have $D^2H = D^2 W + \frac{Q_s}{2R} D^2\F$ and $-L_\F(W) = \tr(D^2 \F(X)^{-1} D^2 W (X)) \leq 0$. So that for $X \in \calC_{-H^-} \subset S_{2R} \setminus Z_0$, 
$$
|\det D^2 H(X)| \leq  \left( \frac{Q_s}{2R}\right)^{n+1} \muF(X),
$$
which combined with \eqref{ABP:H:S2R} gives
\begin{equation}\label{ABP:H:Qs}
\Big(\max\limits_{S_{2R}} H^-\Big)^{n+1} \leq C_n  \left( \frac{Q_s}{2R}\right)^{n+1}  |S_{2R}| \, \muF( \calC_{-H^-}). 
\end{equation}
Now, by \eqref{hyp:W<1}, there exists $X_1=(x_1, z_1) \in S_R$ such that $W(X_1) \leq 1$ and then
\begin{align*}
& H(X_1)\\
&= W(X_1) +  Q_s  \left(\frac{\delta_\F((x_0, 0), X_1)}{2R} - 1 \right) - W_{z,0^+}(2R) |z_1|- \frac{|z_1|}{q_s R^s} + W_{z,0^+}(2R) q_s R^s  + 1\\
& < 1 - \frac{Q_s}{2}   - W_{z,0^+}(2R) |z_1|- \frac{|z_1|}{q_s R^s} + W_{z,0^+}(2R) q_s R^s  + 1\\
& < 3 +  2 q_s W_{z,0^+}(2R)  R^s  - \frac{Q_s}{2},
\end{align*}
where we have used that, from \eqref{z<qsRs}, 
$$
0 \leq 1 - \frac{|z_1|}{q_s R^s} \leq 1 + \frac{|z_1|}{q_s R^s} \leq 2
$$
and 
$$
0 \leq W_{z,0^+}(2R) (q_s R^s - |z_1|) \leq W_{z,0^+}(2R) (q_s R^s + |z_1|) \leq 2 q_s W_{z,0^+}(2R)  R^s.
$$
By the definition of $Q_s$ in \eqref{def:Qs} we get $3 +  2 q_s W_{z,0^+}(2R)  R^s  - \frac{Q_s}{2} = -1$ and then $H^-(X_1) > 1$. 

Then, from \eqref{ABP:H:Qs}, the fact that $H^-(X_1) > 1$ and \eqref{def:K3}, we deduce 
\begin{equation}\label{muFS<mufFC}
\muF(S_{2R}) \leq C_n  K_3 Q_s^{n+1}   \muF( \calC_{-H^-}) . 
\end{equation} 
Next, on the contact set we have $H \leq 0$, so that by using \eqref{W0Rs<eps1} we get (recalling \eqref{01:order>0})
\begin{align*}
 \calC_{-H^-} \subset \{X \in S_{2R} : H(X) \leq 0\} & \subset \{X \in S_{2R} : W(X) \leq Q_s  \}\\
 & \subset  \{X \in S_{2R} : W(X) \leq 4q_s + 8  \}. 
\end{align*}
On the other hand,
\begin{align*}
Q_s^{n+1} = (4 q_s W_{z,0^+}(2R) R^s + 8)^{n+1} \leq   (4 q_s \theta_2 + 8)^{n+1}  \leq  (8 q_s \theta_2)^{n+1} + 16^{n+1},
\end{align*}
so that \eqref{muFS<mufFC} implies
\begin{equation*}
\muF(S_{2R}) \leq C_n  K_3  (8 q_s \theta_2)^{n+1}  \muF(S_{2R}) +  C_n  K_3 16^{n+1} \muF( \calC_{-H^-}) 
\end{equation*}
and by choosing $\theta_2 \in (0,1)$ such that $ C_n  K_3  (8 q_s \theta_2)^{n+1}  \leq 1/2$, the inequality \eqref{muF(W<4qs8} follows with 
\begin{equation*}
\eps_2 := \frac{1}{2 C_n  K_3 16^{n+1}}.
\end{equation*}
\end{proof}

\subsection{The case of sections of $\F$ intersecting $Z_0$}

When a section $S_\F(X_0, R)$, not necessarily centered somewhere at  $\rn \times Z_0$, satisfies $\overline{S_\F(X_0, R)} \cap Z_0 \neq \emptyset$, the first step will be to relate it to a section centered at $\rn \times Z_0$ of comparable height. More precisely, we have

\begin{lem}\label{lem:betas} Given a section $S_\F(X_0, R)$ centered at $X_0=(x_0, z_0)$ with  $\overline{S_\F(X_0, R)} \cap Z_0 \neq \emptyset$, there exists $R_r \in (R, 2K_s R)$ such that
\begin{equation*}
S_\F(X_0, R) \subset S_\F((x_0, 0), 2R_r) \subset  S_\F((x_0, 0), 4R_r) \subset  S_\F(X_0, \beta_s R/2),
\end{equation*}
where $K_s \geq 1$, depending only on $s$, is the quasi-triangle constant for $h_s$ in \eqref{def:Ks} and $\beta_s$, also depending only on $s$, is as in \eqref{def:Ms:betas}.
\end{lem}

\begin{proof}
Given  $S_\F(X_0, R)$ with $\overline{S_\F(X_0, R)} \cap Z_0 \neq \emptyset$, the first inclusion in \eqref{S:vs:SxS}  gives
$$
S_\F(X_0, R) \subset S_\f(x_0, R) \times S_{h_s}(z_0, R).
$$
The fact that $\overline{S_\F(X_0, R)} \cap Z_0 \neq \emptyset$ implies that the closed interval $\overline{S_{h_s}(z_0, R)} \subset \re$ contains $0$, that is, $S_{h_s}(z_0, R) =(z_l, z_r)$ with $z_l \leq 0 \leq z_r$. Without loss of generality, let us assume that $z_r \geq |z_l|$. By putting $R_r:=h_s(z_r)$ we have $(-z_r, z_r) = S_{h_s}(0, R_r)$ (see Figure \ref{fig:sectionshs}) and then
$$
S_{h_s}(z_0, R) =(z_l, z_r) \subset (-z_r, z_r) = S_{h_s}(0, R_r).
$$
\begin{center}
\begin{figure}[htp]
\includegraphics[width=4.5in]{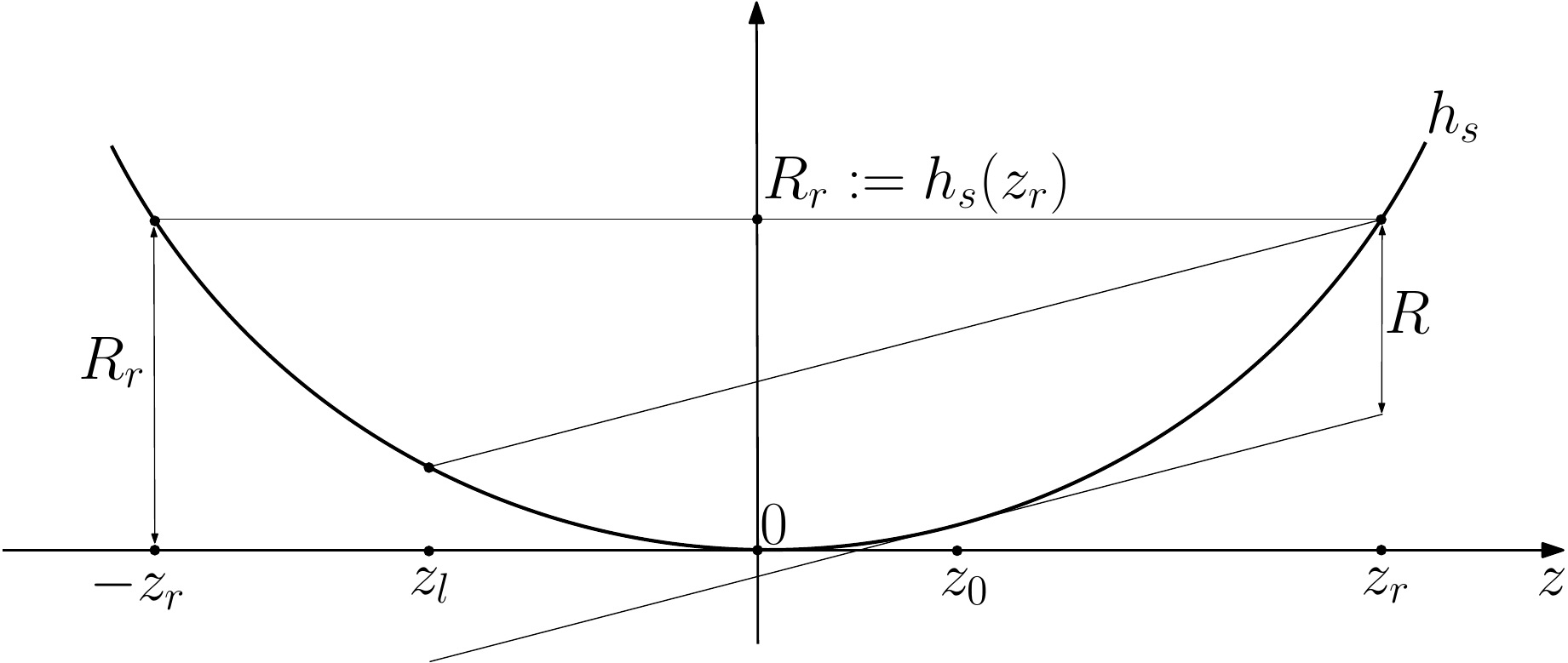}
\caption{On the inclusion $S_{h_s}(z_0, R) =(z_l, z_r) \subset (-z_r, z_r) = S_{h_s}(0, R_r)$.}\label{fig:sectionshs}
\end{figure}
\end{center}
Also, from the quasi-triangle inequality for $h_s$ in \eqref{def:Ks} along with $\delta_{h_s}(z_0, z_r)=R$ and $\delta_{h_s}(z_0, 0) \leq R$, we get
\begin{equation}\label{Rr<R}
R_r = \delta_{h_s}(0, z_r) \leq K_s (\delta_{h_s}(z_0, z_r) +\delta_{h_s}(z_0, 0)  ) \leq 2 K_s R,
\end{equation}
and, using that $R < R_r$,
\begin{align*}
S_\F(X_0, R)  \subset S_\f(x_0, R) \times S_{h_s}(z_0, R) \subset S_\f(x_0, R_r) \times S_{h_s}(0, R_r)  \subset S_\F((x_0, 0), 2R_r).
\end{align*}
Next, we claim that 
\begin{equation}\label{Sh4Rr:Sz0R}
S_{h_s}(0, 4R_r) \subset S_{h_s}(z_0, K_s(1 + 8 K_s) R).
\end{equation}
Indeed, given $z' \in S_{h_s}(0, 4R_r)$ and using the triangle inequality \eqref{def:K3} along with $\delta_{h_s}(z_0, 0) \leq R$ and \eqref{Rr<R},
$$
\delta_{h_s}(z_0, z') \leq K_s ( \delta_{h_s}(z_0, 0) + \delta_{h_s}(0, z') ) < K_s (R + 4 R_r) \leq K_s (1 + 8 K_s) R,
$$
which proves \eqref{Sh4Rr:Sz0R}. Consequently,
\begin{align*}
& S_\F((x_0, 0), 4R_r)   \subset  S_\f(x_0, 4 R_r) \times S_{h_s}(0, 4R_r) \\
& \subset   S_\f(x_0,K_s(1 + 8 K_s) R) \times  S_{h_s}(z_0, K_s(1 + 8 K_s) R) \subset S_\F(X_0, 2 K_s(1 + 8 K_s) R),
\end{align*}
and the lemma is proved. 
\end{proof}

\begin{thm}\label{thm:CD:case3}
Let $\F$ be as in \eqref{def:Phi:s} and $\beta_s$ be as in \eqref{def:Ms:betas} (which is the same constant as in Lemma \ref{lem:betas}). Let $X_0 \in \re^{n+1}$ and $R > 0$ such that $S_\F(X_0, \beta_s R) \subset \subset Q$. Put $S_R:=S_\F(X_0, R)$, $S_{\beta_s R}:= S_\F(X_0, \beta_s R)$, and suppose that  $\overline{S_\F(X_0, 2R)} \cap Z_0 \neq \emptyset$.

Then, there exist geometric constants $\theta_3, \eps_3 \in (0,1)$ such that for every $W \in \calS(Q)$ with $L_\F(W) = - \tr{( (D^2 \F)^{-1} D^2 W})  \geq 0$ and $W \geq 0$ in $S_{\beta_s R} \setminus Z_0$ the inequalities
\begin{equation}\label{W0Rs<theta3}
W_{z,0^+}(\beta_s R) R^s \leq \theta_3
\end{equation}
and
\begin{equation}\label{hyp:W<1:case3}
\inf\limits_{S_R} W \leq 1
\end{equation}
imply
\begin{equation}\label{muF(W<4qs8:case3}
\muF(\{X \in S_{\beta_s R}: W(X) < 4 q_s + 8 \}) \geq \eps_3 \muF(S_{\beta_s R}).  
\end{equation}
\end{thm}

\begin{proof}
Let us take $\theta_3:= (8K_s)^{-s} \theta_2$ with $\theta_2 \in (0,1)$ as in Theorem \ref{thm:CD:center:Z0}. By Lemma \ref{lem:betas} used with $2R$ and setting $\widehat{R}:=(2R)_r \in (2 R, 4K_s R)$, we have
\begin{equation}\label{incl:R:Rr:betaR}
S_\F(X_0, 2R) \subset S_\F((x_0, 0), 2 \widehat{R}) \subset  S_\F((x_0, 0), 4 \widehat{R}) \subset  S_\F(X_0, \beta_s R),
\end{equation}
since, by definition, $\beta_s := 4 K_s(1+8K_s)$.  In particular, by hypothesis \eqref{W0Rs<theta3} and the definition of $\theta_3$,
\begin{equation}\label{W4Rr<theta1}
W_{z,0^+}(4\widehat{R}) (2\widehat{R})^s \leq W_{z,0^+}(\beta_s R) (8K_s R)^s \leq (8K_s)^s \theta_3 = \theta_2.
\end{equation}
Now, inequality \eqref{hyp:W<1:case3} and the first inclusion in \eqref{incl:R:Rr:betaR} imply that $\inf\limits_{S_\F((x_0, 0), 2\widehat{R})} W \leq 1$, which combined with \eqref{W4Rr<theta1} allows us to use Theorem \ref{thm:CD:center:Z0} with the sections $ S_\F((x_0, 0), 2\widehat{R})$ and $S_\F((x_0, 0), 4 \widehat{R})$ to obtain
\begin{equation}\label{CD:2Rr:4Rr}
\muF(\{X \in S_\F((x_0, 0), 4\widehat{R}): W(X) < 4 q_s + 8 \}) \geq \eps_2 \muF(S_\F((x_0, 0), 4\widehat{R})).  
\end{equation}
By the doubling property \eqref{doubling:r:R} and the first inclusion in \eqref{incl:R:Rr:betaR} we get
\begin{equation}\label{muFBsR:4Rr}
\muF(S_\F(X_0, \beta_s R)) \leq K_d \beta_s^\nu \muF(S_\F(X_0, R)) \leq K_d \beta_s^\nu \muF(S_\F((x_0, 0), 4\widehat{R})),
\end{equation}
and then \eqref{muF(W<4qs8:case3} follows, with $\eps_3:= \eps_2 (K_d \beta_s^\nu)^{-1}$, from \eqref{CD:2Rr:4Rr}, \eqref{muFBsR:4Rr}, and the last inclusion in \eqref{incl:R:Rr:betaR}.
\end{proof}

\subsection{Proof of Theorem \ref{thm:CD}}

The proof follows, with the geometric constants $\theta_0:=\theta_3$ and $\eps_0:=\eps_3$, from Theorems \ref{thm:CD:Case1}  and \ref{thm:CD:case3}, and the fact that $\beta_s > 2$ and $M_s > 4$. \qed

\section{Local boundedness}\label{sec:LB}

Let us define
\begin{equation*}
\beta_K:= \max\{K, \beta_s\},
\end{equation*}
where $K$ is as in \eqref{def:K} and $\beta_s$ is as in \eqref{def:Ms:betas}. With the critical-density estimate from Theorem \ref{thm:CD} at hand, we can deduce the following local-boundedness results.

\begin{lem}\label{lem:Lemma6}
Let $\F$ be as in \eqref{def:Phi:s}. There exist geometric constants $N_1, N_2, N_3 > 0$ such that for every  $X_0 \in \re^{n+1}$ and $R > 0$ with $S_\F(X_0, \beta_K R) \subset \subset Q$ and every $W \in \calS(Q)$ with  $L_\F W =  - \tr{( (D^2 \F)^{-1} D^2 W}) \leq 0$ and $W \geq 0$ pointwise in $Q^+$ the inequalities
\begin{equation*}
 R^s  W_{z,0^+}(S_\Phi(X_0,\beta_s R)) \leq N_1 \quad \text{ and } \quad \fint_{S_\F(X_0, 2 K R)} W \dmuF\leq N_2
\end{equation*}
imply
\begin{equation*}
\sup\limits_{S_\F(X_0, R/2)} W \leq N_3. 
\end{equation*}
\end{lem}

\begin{thm}\label{thm:LB:q=1}
Let $\F$ be as in \eqref{def:Phi:s}.  There exist geometric constants $N_4, N_5 > 0$ such that for every  $X_0 \in \re^{n+1}$ and $R > 0$ with $S_\F(X_0, \beta_K R) \subset \subset Q$ and every $W \in \calS(Q)$ with  $L_\F W =  - \tr{( (D^2 \F)^{-1} D^2 W}) \leq 0$ and $W \geq 0$ pointwise in $Q^+$ we have
\begin{equation*}
\sup\limits_{S_\F(X_0, R/2)} W \leq N_4 \fint_{S_\F(X_0, 2 K R)} W \dmuF + N_5  R^s W_{z,0^+}(S_\Phi(X_0,\beta_s R)).
\end{equation*}
\end{thm}

\begin{thm}\label{thm:LB}
Let $\F$ be as in \eqref{def:Phi:s}. Then there are geometric constants $\kappa \in (0,1)$ and $K_4 \geq K$  such that for every $q >0$ there exist constants $C_{1,q}, C_{2,q}> 0$, depending only on geometric constants and $q$, such that for every section $S_\F(X_0,K_4 R) \subset \subset Q$ and every $W \in \calS(Q)$ with $L_\F W =  - \tr{( (D^2 \F)^{-1} D^2 W}) \leq 0$ and $W \geq 0$ pointwise in $Q^+$, we have
\begin{align*}\label{meanvalq}
\sup\limits_{S_\F(X_0,\kappa R)} W & \leq  C_{1,q} \left(\fint\limits_{S_\F(X_0,  R)} W^q  \dmuF \right)^{1/q} + C_{2,q} R^s W_{z,0^+}(S_\F(X_0,K_4 R)).
\end{align*}
\end{thm}

\begin{proof}[About the proofs]
The proof of Lemma \ref{lem:Lemma6} follows as the proof of Lemma 6 in  \cite{MaldonadoJDE14}. Essentially, the only modification is to have the expression 
$$
\frac{t}{ \muf(S_\f(z,t))^{1/n}}\|f\|_{L^n(S_\f(z, 2Kt), \dmuf}
$$ 
(in the proof of \cite[Lemma 6]{MaldonadoJDE14}) replaced with $W_{z,0^+}(S_{\beta_s R}) R^s$ (and, of course, replacing the function $\f$ with $\F$). A key point in the proof is that the sub-solution be locally bounded, which is the case in Lemma \ref{lem:Lemma6} since $W \in \calS(Q)$.

Similarly, Theorem \ref{thm:LB:q=1} follows from  Lemma \ref{lem:Lemma6} as in the proof of Theorem 3 on  \cite[p.2004]{MaldonadoJDE14} and Theorem \ref{thm:LB} follows from Theorem \ref{thm:LB:q=1} as in the proof of Theorem 7 on \cite[pp.2005-8]{MaldonadoJDE14}.
\end{proof}

\section{An arbitrarily sensitive critical-density estimate}\label{sec:every:CD}

By relying on the divergence-form side of the linearized Monge--Amp\`ere operator, in this section we extend Theorem \ref{thm:CD} by proving that every $\eps \in (0,1)$ can work as a critical-density parameter (see Theorem \ref{thm:every:CD} and Corollary \ref{cor:every:CD} below). 

All a.e. statements are referred to Lebesgue measure, which is equivalent to a.e. with respect to $\muF$ due the hypothesis $\f \in C^3(\rn)$ with $D^2 \f > 0$ in $\rn$ and the fact that $\mu_{h_s}(z) = |z|^{1/s-2}$. 

\begin{lem}\label{lem:energy:log:H} 
Let $\F$ be as in \eqref{def:Phi:s}. Fix a section $S_\F(X_0, 2R)$ and suppose that $H \in C(S_\F(X_0, 2R))$ satisfies the following conditions:
\begin{enumerate}[(i)] 
\item $D^2H$ exists a.e. in $S_\F(X_0, 2R)$,
\item $L_\F(H) =  - \tr{( (D^2 \F)^{-1} D^2 H}) \geq 0$ a.e. in $S_\F(X_0, 2R)$,
\item there exists $\tau > 0$ such that $H(X) \geq \tau$ for every $X \in S_\F(X_0, 2R)$,
\item $\nabla^\F H \in L^2(S_\F(X_0, 2R), d\muF)$, that is,
\begin{equation}\label{H:W2F}
\int_{S_\F(X_0, 2R)}  \langle A_\F \nabla H, \nabla H\rangle \dX  < \infty.
\end{equation}
\end{enumerate}
Then,
\begin{equation}\label{energy:log:H}
\fint\limits_{S_\F(X_0,  R)}  |\nFs \log H |^2 \dmuF \leq \frac{32 (n+2)K_d^2}{R}, 
\end{equation}
where $K_d > 1$ is the doubling constant from \eqref{def:Kd}. 
\end{lem}

\begin{proof}
Multiply the inequality $L_\F(H) \geq 0$ by $\muF$ to obtain, a.e. in $S_\F(X_0, 2R)$,
\begin{equation}\label{tr:to:div}
0 \geq \muF \tr((D^2 \F)^{-1} D^2 H) =  \tr(A_\F D^2 H) = \dive (A_\F \nabla H).
\end{equation}
For $\gamma \in C^1(\re)$, supported in $[-2, 2]$,  with $\gamma \equiv 1$ on $[0,1]$ and $\|\gamma'\|_{L^\infty(\re)} \leq 2$  define
\begin{equation*}\label{def:zeta}
\zeta(X):= \gamma \left(\frac{\delta_\F(X_0, X)}{R} \right) \quad \forall X \in \rN.
\end{equation*}
Now multiply \eqref{tr:to:div} by $\zeta^2/H$ and integrate by parts on $S_\F(X_0, 2R)$ to obtain
\begin{align*}
\int_{S_\F(X_0, 2R)}  \frac{\zeta^2 }{H^2} \langle A_\F \nabla H, \nabla H\rangle \dX \leq  \int_{S_\F(X_0, 2R)} \frac{2 \zeta}{H} \langle A_\F \nabla H, \nabla \zeta \rangle \dX.
\end{align*}
From Young's inequality, 
\begin{align*}
&\int_{S_\F(X_0, 2R)} \frac{2 \zeta}{H} \langle A_\F \nabla H, \nabla \zeta \rangle \dX\\
& \leq \frac{1}{2} \int_{S_\F(X_0, 2R)}  \frac{\zeta^2 }{H^2} \langle A_\F \nabla H, \nabla H\rangle \dX  + 2 \int_{S_\F(X_0, 2R)}  \langle A_\F \nabla \zeta, \nabla \zeta \rangle \dX.
\end{align*}
Hence,
\begin{equation*}
\int_{S_\F(X_0, 2R)}  \frac{\zeta^2 }{H^2} \langle A_\F \nabla H, \nabla H\rangle \dX \leq   4 \int_{S_\F(X_0, 2R)}  \langle A_\F \nabla \zeta, \nabla \zeta \rangle \dX,
\end{equation*}
where we have used that $\zeta^2 \leq 1$, $0 < \tau \leq H$, and \eqref{H:W2F} to guarantee that
$$
\int_{S_\F(X_0, 2R)}  \frac{\zeta^2 }{H^2} \langle A_\F \nabla H, \nabla H\rangle \dX < \infty.
$$
On the other hand, since
$$
\nabla \zeta(X) = \frac{1}{R} \gamma' \left(\frac{\delta_\F(X_0, X)}{R} \right) (\nabla \F(X) - \nabla \F(X_0))
$$
we get
\begin{align*}
& \int_{S_\F(X_0, 2R)}  \langle A_\F \nabla \zeta, \nabla \zeta \rangle \dX\\
& \leq \frac{4}{R^2} \int_{S_\F(X_0, 2R)}  \langle A_\F(X) (\nabla \F(X) - \nabla \F(X_0)), (\nabla \F(X) - \nabla \F(X_0)) \dX\\
& \leq \frac{8(n+2)K_d}{R} \muF(S_\F(X_0, 2R)),
\end{align*}
where for the last inequality we used the energy estimate \eqref{L2:Energy:delta:Phi}. Therefore,
\begin{align*}
& \int\limits_{S_\F(X_0,  R)}  |\nFs \log H |^2 \dmuF  = \int_{S_\F(X_0, R)}  \frac{1 }{H^2} \langle A_\F \nabla H, \nabla H\rangle \dX\\
& \leq \int_{S_\F(X_0, 2R)}  \frac{\zeta^2 }{H^2} \langle A_\F \nabla H, \nabla H\rangle \dX \leq \frac{32 (n+2)K_d}{R} \muF(S_\F(X_0, 2R)),
\end{align*}
and \eqref{energy:log:H} follows from the doubling property \eqref{def:Kd}.
\end{proof}

\begin{thm}\label{thm:every:CD}
Let $\F$ be as in \eqref{def:Phi:s}.  Fix $W \in \calS(Q)$ with $W \geq  0$ and $L_\F W =  - \tr{( (D^2 \F)^{-1} D^2 W}) \geq 0$  in $Q^+$. Fix $X_0 \in \re^{n+1}$ and $R>0$ with $S_\F(X_0,K_4 R) \subset \subset Q$ and put $S:= S_\F(X_0,R)$. 

Then, given $\eps, \tau \in (0,1)$, the inequalities
\begin{equation}\label{Wz0:K4R:tau}
W_{z,0^+}(S_{K_4 R}) R^s \leq \tau
\end{equation}
and
\begin{equation}\label{blabla}
\muF( \{X \in S : W(X) \geq 1- \tau \}) \geq \eps \muF(S)
\end{equation}
imply that
\begin{equation}\label{infW>C}
\inf\limits_{S_\F(X_0,\kappa \,R)} W + \tau \geq  e^{-C_{0}(\eps)},
\end{equation}
where $\kappa \in (0,1)$ and $K_4 >1$ are the geometric constants from Theorem \ref{thm:LB} and
\begin{equation}\label{def:C0}
C_0(\eps):=  C_{1,1}  K_P K_d \left(1+ \frac{1}{\eps}\right)  \sqrt{ \frac{32 (n+2)}{K_2}} + C_{2,1}, 
\end{equation}
where $K_P, K_2 > 0$ are the geometric constants from the weak Poincar\'e inequality  \eqref{Poincare:Phi} and  $C_{1,1}, C_{2,1} >0$ are the geometric constants from Theorem \ref{thm:LB} corresponding to $q=1$. 
\end{thm}

\begin{proof}
For $\tau \in (0,1)$ define the function $\ell_\tau : [0, \infty) \to [0, \infty)$ as 
\begin{equation*}\label{def:ltau}
\ell_\tau(t):= \left\{
      \begin{array}{lcl}
       -\frac{1}{\tau}(t - \tau) - \log (\tau) &  \text{if } & 0 < t < \tau, \\
           -\log (t) &  \text{if } &  \tau \leq t < 1,\\
           0 &\text{if } &   1 \leq t,
      \end{array}
    \right.
\end{equation*}
which is a convex, Lipschitz function with $\ell_\tau' \leq 0$ and $\|\ell_\tau'\|_{L^\infty[0,\infty)} \leq 1/\tau$, and put
\begin{equation}\label{def:G}
G(X) := \ell_{\tau}(W(X) + \tau) \quad \forall X \in S_\F(X_0, K_4 R).
\end{equation}
Let $\{\ell_{\tau,\ep}\}_{\ep \in (0,1)}$ be a smooth approximation of $\ell_\tau$ such that
\begin{equation}\label{prop:ell:tau:ep}
\ell_{\tau,\ep}'' \geq 0, \quad \ell_{\tau,\ep}' \leq 0, \quad \|\ell_{\tau, \ep}'\|_{L^\infty[0,\infty)}\leq 1/\tau \quad \forall \ep \in (0,1)
\end{equation}
and introduce
\begin{equation*}
G^{(\ep)}(X) := \ell_{\tau,\ep}(W(X) + \tau) \quad \forall X \in Q.
\end{equation*}
Let us see that $G^{(\ep)} \in \calS(Q)$ for every $\ep \in (0,1)$ by checking the conditions \eqref{z:symmetry}-\eqref{normal:derivative} from the definition of $\calS(Q)$ in Section \ref{sec:S(Q)}. Conditions \eqref{z:symmetry} and \eqref{C:C2} are immediate. Condition \eqref{finite:L2:E} follows from
$$
\nabla^\F G^{(\ep)} = \ell_{\tau,\ep}'(W(X) + \tau) \nabla^\F W(X) \quad \forall X \in Q
$$
and the fact that $\|\ell_{\tau, \ep}'\|_{L^\infty[0,\infty)} \leq 1/\tau$, uniformly in $\ep \in (0,1)$. Similarly, condition \eqref{normal:derivative}, follows from
$$
G^{(\ep)} _z =  \ell_{\tau,\ep}'(W(X) + \tau)W_z(X) \quad \forall X \in Q.
$$ 
Consequently, $G^{(\ep)} \in \calS(Q)$ for every $\ep \in (0,1)$. 

On the other hand, 
$$
D^2G^{(\ep)} = \ell_{\tau,\ep}''(W + \tau) (\nabla W \otimes \nabla W) + \ell_{\tau,\ep}'(W + \tau) D^2 W \quad \text{in } Q^+
$$
and then, always in $Q^+$, 
\begin{align*}
-L_\F(G^{(\ep)}) &= \tr((D^2 \F)^{-1}D^2G^{(\ep)})\\
& =  \ell_{\tau,\ep}''(W + \tau)  \langle (D^2 \F)^{-1}  \nabla W, \nabla W \rangle +  \ell_{\tau,\ep}'(W + \tau)  \tr((D^2 \F)^{-1}D^2 W) \\
& \geq  \ell_{\tau,\ep}'(W + \tau)  \tr((D^2 \F)^{-1}D^2 W) \geq  0.
\end{align*}
That is, $G^{(\ep)} \in \calS(Q)$ satisfies $L_\F(G^{(\ep)}) \leq 0$ and $G^{(\ep)} \geq 0$  in $Q^+$. By Theorem \ref{thm:LB} applied to $G^{(\ep)}$ with $q=1$ we have
\begin{equation}\label{Uep:meanvalq}
\sup\limits_{S_\F(X_0,\kappa \,R)} G^{(\ep)}  \leq C_{1,1} \fint\limits_{S_\F(X_0,  R)} G^{(\ep)}  \dmuF  + C_{2,1} G^{(\ep)}_{z,0^+}(S_{K_4 R}) R^s.
\end{equation}
And, by \eqref{prop:ell:tau:ep} and \eqref{Wz0:K4R:tau}, 
\begin{align*}
G^{(\ep)}_{z,0^+}(S_{K_4 R}) R^s \leq \frac{1}{\tau} W_{z,0^+}(S_{K_4 R}) R^s \leq 1, \quad \forall \ep \in (0,1),
\end{align*}
and then, by taking limits as $\ep \to 0$ in \eqref{Uep:meanvalq}, it follows that
\begin{equation*}
\sup\limits_{S_\F(X_0,\kappa \,R)} G  \leq  C_{1,1} \fint\limits_{S_\F(X_0,  R)} G  \dmuF + C_{2,1}.
\end{equation*}
Now, since 
$$
\{X \in S : W(X) \geq 1- \tau \} = \{ X \in S : W(X) + \tau \geq 1\} = \{ X \in S: G(X) = 0\},
$$
the hypothesis \eqref{blabla}
says that $\muF( \{ X \in S: G(X) = 0\}) \geq \eps \muF(S)$ and, by Corollary \ref{cor:Fabes}, 
\begin{align*}
& \fint\limits_{S_\F(X_0,  R)} G  \dmuF \leq K_P (1+ \tfrac{1}{\eps}) R^{1/2} \left(\fint\limits_{S_\F(X_0, K_2  R)} |\nF G|^2 \dmuF   \right)^{1/2}\\
 & = K_P (1+ \tfrac{1}{\eps}) R^{1/2}\left(\fint\limits_{S_\F(X_0,  K_2 R)}  |\nF \log( W + \tau)|^2 \dmuF   \right)^{1/2}.
\end{align*}
At this point we use Lemma \ref{lem:energy:log:H} with $H:=W + \tau$ in the section $S_\F(X_0,  K_2 R)$ to get
$$
\fint\limits_{S_\F(X_0,  R)} G  \dmuF \leq K_P K_d \left(1+ \frac{1}{\eps}\right)  \sqrt{ \frac{32 (n+2)}{K_2}}
$$
and then
\begin{equation}\label{def:C0eps}
\sup\limits_{S_\F(X_0,\kappa \,R)} G  \leq C_{1,1}  K_P K_d \left(1+ \frac{1}{\eps}\right)  \sqrt{ \frac{32 (n+2)}{K_2}} + C_{2,1} =: C_{0}(\eps).
\end{equation}
By the definition of $G$ in \eqref{def:G}, we have
\begin{equation}\label{supG:infW}
\sup\limits_{S_\F(X_0,\kappa \,R)} G = \ell_{\tau}\left(\inf\limits_{S_\F(X_0,\kappa \,R)} W + \tau\right)
\end{equation}
and  \eqref{infW>C} follows from \eqref{def:C0eps} and \eqref{supG:infW}. 
\end{proof}

\begin{cor}\label{cor:every:CD}
Let $\F$ be as in \eqref{def:Phi:s}.  Fix $W \in \calS(Q)$ with $W \geq  0$ and $L_\F W =  - \tr{( (D^2 \F)^{-1} D^2 W}) \geq 0$  in $Q^+$. Fix $X_0 \in \re^{n+1}$ and $R>0$ with $S_\F(X_0,K_4 R) \subset \subset Q$ and put $S:= S_\F(X_0,R)$. Then, for every $\eps \in (0,1)$ there exists $\tau \in (0,1)$, depending only on $\eps$ and geometric constants, such that the inequalities
\begin{equation*}\label{hyp:f:tau:2}
W_{z,0^+}(S_\Phi(X_0,K_4 R)) R^s \leq \frac{\tau}{1-\tau}
\end{equation*}
and  
\begin{equation}\label{hyp:meas:U:0}
\muF( \{X \in S : W(X) \geq 1 \}) \geq \eps \muF(S)
\end{equation}
imply that
\begin{equation*}
\inf\limits_{S_\F(X_0,\kappa \,R)} W \geq \frac{\tau}{1-\tau},
\end{equation*}
where $\kappa \in (0,1)$ is the geometric constant from Theorem \ref{thm:LB}.
\end{cor}

\begin{proof}
Given $\eps \in (0,1)$ let $\tau \in (0,1)$ be defined by
\begin{equation}\label{def:tau}
2 \tau:= e^{-C_0(\eps)}
\end{equation}
and apply Theorem \ref{thm:every:CD} to $W_\tau:=(1-\tau)W$. 
\end{proof}

\section{The weak-Harnack inequality}\label{sec:weak:H}

In this section we prove a weak-Harnack inequality for nonnegative super-solutions in $\calS(Q)$ (see Theorem \ref{thm:weak:H}). The proof relies on establishing a ``non-homogeneous'' version of Theorem 7.1 in \cite{KS01}. Towards that end, we next adapt a result known as the ``crawling ink spots lemma'' to the elliptic Monge--Amp\`ere context. The ``crawling ink spots lemma'' has been developed by Krylov-Safonov in \cite[Section 2]{KrSa81};  the lemmas in \cite[Section 2]{KrSa81} correspond to the parabolic case;  see \cite[Lemma 1.1]{Sa80} for the elliptic case in $\rn$, and \cite[Lemma 7.2]{KS01},  for instance, for a version in doubling metric spaces).

\begin{lem}\label{lem:CIS}
Let $\F$ be as in \eqref{def:Phi:s}. Fix any $K_0  > K( 2K + 1)$, where $K \geq 1$ is the quasi-triangle constant from \eqref{def:K}. Given a section $S:=S_\F(X_0, R)$, a measurable subset $E \subset S$, and $\delta \in (0,1)$ define the open set
\begin{equation*}
E_\delta := \bigcup_{0 < \rho < K_0 R} \{S_\F(X,  \rho) \cap S: X \in S \, \text{ and } \, \muF(E \cap S_\F(X, \rho)) > \delta \muf(S_\F(X,\rho)) \}.
\end{equation*}
Then either $E_\delta = S$ or 
\begin{equation}\label{E<Ed}
\muF(E) \leq \delta K_d K_0^\nu  \, \muF(E_\delta),
\end{equation}
where $\nu \geq 1$ is as in \eqref{doubling:r:R}. 
\end{lem}

\begin{proof}
The proof follows as the one for \cite[Lemma 7.2]{KS01} by means of Vitali's covering lemma for Monge--Amp\`ere sections. In turn, Vitali's covering lemma for Monge--Amp\`ere sections follows as in the proof of Theorem 1.2 in \cite[p.69]{CW71} for general spaces of homogeneous type. In the Monge--Amp\`ere case, the dilation constant in Vitali's lemma can be any $K_0$ with $K_0 > 2 K^2 + K$. 
\end{proof}

\begin{thm}\label{thm:weak:H}
Let $\F$ be as in \eqref{def:Phi:s}.  There exist geometric constants $\sigma \in (0,1)$ and $K_6, K_7 > 1$ such that for every $W \in \calS(Q)$ with  $L_\F W =  - \tr{( (D^2 \F)^{-1} D^2 W}) \geq 0$ and $W \geq  0$ in $Q^+$ and every $(X_0, R) \in \re^{n+1} \times(0,\infty)$ with $S_\F(X_0,K_7R) \subset \subset Q$, we have
\begin{align}\label{weak:H:W}
&\left(\fint_{S_\F(X_0,R)}W (X)^\sigma \muF(X) \right)^\frac{1}{\sigma} \leq K_6 \left(\inf\limits_{S_\F(X_0, \kappa R)} W +  R^s W_{z,0+}(S_\F(X_0,K_7R))  \right),
\end{align}
where $\kappa \in (0,1)$ is the geometric constant from Theorem \ref{thm:LB}.
\end{thm}

\begin{proof}
Let us start by defining $K_7 > 1$ as
\begin{equation}\label{def:K7}
K_7 := K (K_4 K_0 + \kappa),
\end{equation}
with $K_4 \geq K$ being the geometric constant from Theorem \ref{thm:LB}. Let us fix $\delta \in (0,1)$ such that 
\begin{equation}\label{def:delta0}
\delta_0 := \delta K_d K_0^\nu < 1,
\end{equation}
where $K_d K_0^\nu$  is the product of geometric constants from \eqref{E<Ed} in Lemma \ref{lem:CIS}, and for $\delta_0$ as in \eqref{def:delta0} choose $\eps \in (0,1)$ such that
\begin{equation}\label{def:eps:CD}
\eps:= \frac{\delta \kappa^\nu}{K_d} < \delta_0.
\end{equation}
With this choice of $\eps$, and recalling the definition of $C_0(\eps)$ in \eqref{def:C0}, define $\tau \in (0,1)$ by means of \eqref{def:tau}
and put $\lambda:= \tau/(1-\tau) \in (0,1)$, which makes $\tau$, $\delta_0$, and $\lambda$ geometric constants. 

For $t > 0$ and $i \in \naz$ set
\begin{equation*}
A_{t, i}:= \{X \in S_\F(X_0, \kappa R): W(X) \geq t \lambda^i\}
\end{equation*}
and let $j = j(t)\in \na$ satisfy
\begin{equation}\label{def:j}
\delta_0^j \leq \frac{\muF(A_{t, 0})}{\muF(S_\F(X_0, R))} \leq \delta_0^{j-1},
\end{equation}
which yields
\begin{equation}\label{At0<j1}
\left(  \frac{\muF(A_{t, 0})}{\muF(S_\F(X_0, R))}  \right)^\gamma \leq \lambda^{j-1}
\end{equation}
where
\begin{equation*}
\gamma:= \frac{\log \lambda}{\log \delta_0}
\end{equation*}
is a geometric constant. We will show that
\begin{equation}\label{tj1<}
t \lambda^{j-1} \leq K_8 \left( \inf\limits_{S_\F(X_0, \kappa R)} W +  R^s W_{z,0+}(S_\F(X_0,K_7R)) \right),
\end{equation}
where $K_8 > 1$ is the geometric constant defined as
\begin{equation}\label{def:K8}
K_8:= K_0^s \lambda^{-1}. 
\end{equation}
Indeed, given $i \in \{1,\ldots, j\}$ we consider two cases: when
\begin{equation}\label{case2:F>}
 K_0^s R^s W_{z,0+}(S_\F(X_0,K_7R))  > t \lambda^i
\end{equation}
and when
\begin{equation}\label{case1:F<}
  K_0^s R^s W_{z,0+}(S_\F(X_0,K_7R)) \leq t \lambda^i.
\end{equation}
If \eqref{case2:F>} holds true for some $i \in \{1,\ldots, j\}$, then \eqref{tj1<} is immediate from the definition of $K_8$ in \eqref{def:K8} and the fact that $\lambda \in (0,1)$ and  $i \in \{1,\ldots, j\}$ imply  $\lambda^i \geq \lambda^{j}$. 

Suppose then that \eqref{case1:F<} holds true for every  $i \in \{1,\ldots, j\}$.  We will prove \eqref{tj1<} by repeatedly applying Corollary \ref{cor:every:CD} and Lemma \ref{lem:CIS}. 

If for some $X \in S_\F(X_0, \kappa R)$ and $\rho \in (0, \kappa K_0  R)$ we have
$$
\muF(A_{t, i-1} \cap S_\F(X, \rho)) > \delta \muF(S_\F(X, \rho)),
$$
(with $\delta$ as in \eqref{def:delta0})
by the doubling property \eqref{doubling:r:R} (and recalling that $\kappa \in (0,1)$) it follows that 
\begin{align*}
\muF(\{Y \in S_\F(X, \rho/\kappa):   \frac{W(Y)}{t \lambda^{i-1}} \geq 1\})& \geq \muF(A_{t, i-1} \cap S_\F(X, \rho)) > \delta  \muF(S_\F(X, \rho))\\
& \geq \frac{\delta \kappa^\nu}{K_d}  \muF( S_\F(X, \rho/\kappa))= \eps  \muF( S_\F(X, \rho/\kappa)),
\end{align*}
where for the last equality we used the definition of $\eps$ in \eqref{def:eps:CD}. Next, let us see that $\rho \in  (0, \kappa K_0  R)$ and $X \in S_\F(X_0, \kappa R)$ imply the inclusion
\begin{equation}\label{incl:SX:SX0}
S_\F(X, K_4 \rho/\kappa) \subset S_\F(X_0, K_7 R).
\end{equation}
Indeed, given $Y \in S_\F(X, K_4 \rho/\kappa)$, by the $K$-quasi-triangle inequality \eqref{def:K}
\begin{align*}
\delta_\F(X_0, Y) \leq K (\delta_\F(X,Y) + \delta_\F(X_0, X)) & < K \left(\frac{K_4 \,\rho}{\kappa} + \kappa R \right)\\
& \leq K \left(K_4 K_0 R + \kappa R \right) = K_7 R,
\end{align*}
where for the last equality we used the definition of $K_7$ in \eqref{def:K7}. Now, from the fact that $\rho/\kappa < K_0 R$, the inclusion \eqref{incl:SX:SX0}, and the hypothesis \eqref{case1:F<}, we get
\begin{align*}
& (\rho/\kappa)^s W_{z,0+}(S_\F(X,  K_4 \rho/ \kappa))  \leq K_0^s R^s  W_{z,0+}(S_\F(X_0, K_7 R)) \leq t \lambda^i = t \lambda^{i-1} \frac{\tau}{1-\tau}.
\end{align*}
Then, Corollary \ref{cor:every:CD} applied to $\frac{W(X)}{t \lambda^{i-1}}$ on the section $S_\F(X, \rho/\kappa)$ with $\eps$ as in \eqref{def:eps:CD} yields 
\begin{equation*}
\inf\limits_{S_\F(X, \rho)} W \geq t \lambda^i
\end{equation*}
and, consequently, $S_\F(X_0, \kappa R) \cap S_\F(X, \rho) \subset A_{t, i}$. By Lemma \ref{lem:CIS} applied to the section $S_\F(X_0, \kappa R)$ and the set $E:= A_{t, i-1}$ it follows that either $A_{t, i-1} = S_\F(X_0, \kappa R)$ or
\begin{equation}\label{Ai1Ai}
\frac{1}{\delta_0} \muF(A_{t, i-1}) \leq \muF(E_\delta) \leq \muF(A_{t, i}),
\end{equation}
where $\delta_0 \in (0,1)$ is as in \eqref{def:delta0}. Now, if $A_{t, i-1} = S_\F(X_0, \kappa R)$  for some  $i \in \{1,\ldots, j\}$, then (due to the inclusion $A_{t, i-1} \subset A_{t, j-1}$) we have  $A_{t, j-1} = S_\F(X_0, \kappa R)$, which means
$$
\inf\limits_{S_\F(X_0, \kappa R)} V \geq t \lambda^{j-1},
$$
and the inequality \eqref{tj1<} follows. Hence, we can assume that \eqref{Ai1Ai} holds for every $i \in \{1, \ldots, j\}$ and then write
\begin{align*}
\muF(A_{t, j-1}) \geq \frac{1}{\delta_0} \muF(A_{t, j-2})  \geq \cdots \geq \frac{\muF(A_{t, 0}) }{\delta_0^{j-1}} \geq \delta_0 \, \muF(S_\F(X_0, \kappa R),
\end{align*}
where for the last inequality we used the definition of $j$ in \eqref{def:j}. In particular,
\begin{align*}
\muF(\{X \in S_\F(X_0, \kappa R): W(X) \geq t \lambda^{j-1}\})& = \muF(A_{t, j-1}) \geq \delta_0 \, \muF(S_\F(X_0, \kappa R)\\
& > \eps \, \muF(S_\F(X_0, \kappa R))
\end{align*}
and, using that $\kappa < K_0$ and $\kappa K_4 < K_7$, by \eqref{case1:F<} applied with $i=j$, we obtain
\begin{align*}
& (\kappa R)^s W_{z,0^+}(S_\F(X_0, \kappa K_4  R)) \leq K_0^s R^s W_{z,0^+}(S_\F(X_0, K_7  R)) < t \lambda^j = t \lambda^j \frac{\tau}{1-\tau},
\end{align*}
so that Corollary \ref{cor:every:CD} applied to $\frac{W(X)}{t \lambda^{j-1}}$  on the section $S_\F(X_0, \kappa R)$ with $\eps$ as in \eqref{def:eps:CD} yields 
\begin{equation*}
\inf\limits_{S_\F(X_0, \kappa R)} W \geq t \lambda^j
\end{equation*}
and \eqref{tj1<} follows. Now, by setting $\xi: = K_8 ( \inf\limits_{S_\F(X_0, \kappa R)} W + R^s W_{z,0+}(S_\F(X_0,K_7R)))$, from \eqref{At0<j1} and \eqref{tj1<} we obtain
\begin{equation*}
\muF(\{X \in S_\F(X_0, \kappa R): W(X) \geq t \}) = \muF(A_{t,0}) \leq \left(\frac{\xi}{t}\right)^\frac{1}{\gamma} \quad \forall t > 0 
\end{equation*}
and then, for any $\sigma \in (0, 1/\gamma)$,
\begin{align*}
&\fint_{S_\F(X_0, \kappa R)} W(X)^\sigma \dmuF(X)\\
& \leq \sigma \int_0^\xi t^{\sigma -1} \, dt +  \frac{\sigma}{\muF(S_\F(X_0, \kappa R))} \int_\xi^\infty t^{\sigma-1}  \muF(A_{t,0}) \, dt \leq \left(\frac{\gamma}{1- \sigma \gamma} + 1 \right)\xi^\sigma,
\end{align*}
which proves \eqref{weak:H:W} with $K_6:= \left(\frac{\gamma}{1- \sigma \gamma} + 1 \right)^{1/\sigma} K_8$. 
\end{proof}

\section{Proofs of Theorems \ref{H:nonDiver} and \ref{thm:H:Vtilde}}\label{sec:proofs:main}

We are now position to prove Theorems \ref{H:nonDiver} and \ref{thm:H:Vtilde}. Let us start with

\begin{proof}[Proof of Theorem  \ref{thm:H:Vtilde}.]
By Theorem \ref{thm:weak:H} applied to $W \in \calS(Q)$, we get
\begin{align*}\label{weak:W:V}
&\left(\fint_{S_\F(X_0,R)} W(X)^\sigma \muF(X) \right)^\frac{1}{\sigma} \leq K_6 \left(\inf\limits_{S_\F(X_0, \kappa R)} W +  R^s W_{z,0+}(S_\F(X_0,K_7R))  \right).
\end{align*}
Now, by Theorem \ref{thm:LB} applied to $W$ with $q=\sigma$,
\begin{align*}
\sup\limits_{S_\F(X_0,\kappa \,R)} W& \leq  C_{1,\sigma} \left(\fint\limits_{S_\F(X_0,  R)} W(X)^\sigma  \dmuF(X) \right)^{\frac{1}{\sigma}} +  C_{2,\sigma}  R^sW_{z,0+}(S_\F(X_0,K_4R)),
\end{align*}
and notice that, from \eqref{def:K7}, we have $K_4 \leq K_7$. Hence, the Harnack inequality \eqref{H:W} follows with $\widetilde{C}_H:= C_{1,\sigma} K_6 + C_{2, \sigma}$. 

Next, for $0 < r < R$, consider the functions 
$$
M(r):= \sup\limits_{S_\F(X_0,r)} W \quad \text{and} \quad m(r):= \inf\limits_{S_\F(X_0,r)} W.
$$
A standard argument (see for instance \cite[Section 8.9]{Gilbarg-Trudinger}) implies the existence of geometric constants $\varrho \in (0,1)$ and $K_{11} >0$ such that
\begin{equation}\label{Mr:mr}
M(r) - m(r) \leq K_{11}  r^\varrho \left(\sup\limits_{S_\F(X_0,R)} W + R^sW_{z,0+}(S_\F(X_0,K_4R)\right), \quad \forall r \in (0, R),
\end{equation}
which, in turn, implies \eqref{Holder:W}. 
\end{proof}

\begin{proof}[Proof of Theorem  \ref{H:nonDiver}.]
Given a section $S_0:=S_\f(p_0, R_0)$,  $f\in C_0(\overline{S_0})$, and a nonnegative  $v\in\Dom_{S_0}(L_\f)$ solution to $L_\f^sv=f$ in $S_0$, let $V$ be the solution to the extension problem \eqref{eq:extension nondivergence form}. In particular, we have $V \in C_0(\overline{Q})$ and $\lim_{z\to0^+}V(x,z)=v(x)$ uniformly in $S_0$. 

Now, let us set $K_9:=2 K_7$ so that given a section $S:=S_\f(x_0, R)$ with $S_\f(x_0, K_9 R) \subset \subset S_0$, by the first inclusion in \eqref{S:vs:SxS} we get
\begin{align*}
S_\F((x_0, 0), 2 K_7 R) & \subset S_\f(x_0, 2 K_7 R) \times S_{h_s}(0, 2K_7 R)\\
& =  S_\f(x_0, K_9 R) \times S_{h_s}(0, K_9 R) \subset \subset S_0 \times \re = Q,
\end{align*}
and then the Harnack inequality \eqref{H:W} for $\widetilde{V}$ on the section $S_\F((x_0, 0), 2R)$ gives
\begin{equation}\label{H:Vtilde:0}
\sup\limits_{S_\F((x_0, 0), 2 \kappa R)} \widetilde{V} \leq \widetilde{C}_H \left( \inf\limits_{S_\F((x_0, 0), 2 \kappa R)} \widetilde{V} + R^s  \widetilde{V}_{z,0+}(S_\F((x_0,0), 2 K_7R)) \right).
\end{equation}
On the other hand, the second inclusion in \eqref{S:vs:SxS} yields
\begin{align*}
S_\f(x_0, \kappa R) \times \{0\} \subset S_\f(x_0, \kappa R) \times S_{h_s}(0, \kappa R) \subset S_\F((x_0, 0), 2 \kappa R)
\end{align*}
and, since $\widetilde{V}(x, 0) = V(x,0) = v(x)$ for every $x \in S_0$, it follows that
$$
\sup\limits_{S_\f(x_0, \kappa R)} v =  \sup\limits_{S_\f(x_0, \kappa R)} \widetilde{V}(\cdot, 0) 
\leq \sup\limits_{S_\F((x_0, 0), 2 \kappa R)} \widetilde{V}
$$ 
as well as
$$
\inf\limits_{S_\F((x_0, 0), 2 \kappa R)} \widetilde{V} \leq \inf\limits_{S_\f(x_0, \kappa R)} \widetilde{V}(\cdot, 0) = \inf\limits_{S_\f(x_0, \kappa R)} v,  
$$ 
which, along with \eqref{H:Vtilde:0} implies
\begin{equation}\label{H:v:Vz}
\sup\limits_{S_\f(x_0, \kappa R)} v \leq \widetilde{C}_H \left(\inf\limits_{S_\f(x_0, \kappa R)} v + R^s  \widetilde{V}_{z,0+}(S_\F((x_0,0), 2 K_7R)) \right).
\end{equation}
But, by the inclusion $S_\F((x_0, 0), 2 K_7 R)  \subset S_\f(x_0, 2 K_7 R) \times S_{h_s}(0, 2K_7 R)$, we have
$$
S_\F((x_0, 0), 2 K_7 R) \cap Z_0 \subset S_\f(x_0, 2 K_7 R) \times \{0\},
$$
so that from the definition of $\widetilde{V}_{z,0+}(S_\F((x_0,0), 2 K_7R)) $ in \eqref{def:F0+} and the limit in \eqref{Vz0=Lv}, we get
\begin{equation}\label{Vz0:Linfty}
\widetilde{V}_{z,0+}(S_\F((x_0,0), 2 K_7R))  \leq d_s \|L_\f^sv\|_{L^\infty(S_\f(x_0, 2 K_7R))}.
\end{equation}
Then, the Harnack inequality \eqref{H:v} follows from \eqref{H:v:Vz} and \eqref{Vz0:Linfty} with $K_9:=2 K_7$ and $C_H:= d_s \widetilde{C}_H$.

Also by a restriction argument, the Monge--Amp\`ere H\"older estimate \eqref{Holder:v} follows from \eqref{Holder:W} applied to $\widetilde{V}$ and then restricting to $Z_0$. Notice that in principle we cannot prove \eqref{Holder:v} directly from the Harnack inequality \eqref{H:v}, as we did to go from \eqref{H:W} to \eqref{Holder:W} via \eqref{Mr:mr}. This is due to the fact that if $v \in \Dom_S(L_\f)$ and $C \in \re \setminus\{0\}$, then it is not true that $v - C\in \Dom_S(L_\f)$ because $v-C$ does not vanish on $\partial S$. 
\end{proof}

\section*{Acknowledgements}

The authors would like to thank the referee for suggestions that helped to improve the presentation of the paper. The first author was supported by NSF under grant DMS 1361754. The second author was partially supported by grant MTM2015-66157-C2-1-P (MINECO/FEDER) from Government of Spain.



\end{document}